\documentclass[a4paper,12pt,oneside]{amsart}

\usepackage{amssymb}  % defines \nmid, for example
\usepackage{latexsym} % for \Box
\usepackage{comment}
\usepackage{color}
\usepackage{colonequals}

 \usepackage{epsfig}
\usepackage{tikz}
\usepackage{tikz-cd}

\definecolor{darkgreen}{rgb}{0,0.5,0}
\usepackage[
        colorlinks, citecolor=darkgreen,
        pdfauthor={Ingrid Bauer, Roberto Pignatelli},
        pdftitle={Rigid, non infinitesimally rigid surfaces},
        linktocpage        
]{hyperref}

\usepackage[alphabetic,backrefs,lite]{amsrefs} % for bibliography

\usepackage[all]{xy} % added "lite" to amsrefs and moved xy here; hope this works

\usepackage{enumerate} % for flexible labels in enumerated lists

\usepackage{listings}

\lstdefinelanguage{Magma}
{
keywords={for,end,if,then,else,elif,while,function,return,cat,&,and,or,do,eq,ne,ge,le,gt,lt,mod, meet,in,notin  },
morekeywords={Seqset,Setseq,Polytope,AutomorphismGroup,RowSequence,IdentifyGroup,
	      Subgroups,PermutationMatrix,Generators,MatrixGroup,Transpose},
sensitive=false,
morecomment=[l]{//},
morecomment=[s]{/*}{*/},
morestring=[b]",
}

\lstnewenvironment{codice_magma}[1][]
{\lstset{basicstyle=\small\ttfamily, columns=fullflexible,
language=Magma,
basicstyle=\scriptsize,
keywordstyle=\color{red}\bfseries,
commentstyle=\color{blue},tabsize=4,
numbers=left, numberstyle=\tiny,
stepnumber=2, numbersep=5pt, frame=single, #1}}{}

\newcommand\sI{\ensuremath{{\mathcal I}}}

\newcommand\sL{\ensuremath{{\mathcal L}}}

\newcommand\sO{\ensuremath{{\mathcal O}}}
\newcommand\sP{\ensuremath{{\mathcal P}}}

\newcommand{\CC}{\ensuremath{\mathbb{C}}}

\newcommand{\NN}{\ensuremath{\mathbb{N}}}
\newcommand{\PP}{\ensuremath{\mathbb{P}}}
\newcommand{\QQ}{\ensuremath{\mathbb{Q}}}

\newcommand{\ZZ}{\ensuremath{\mathbb{Z}}}

\DeclareMathOperator{\Aut}{Aut}

\DeclareMathOperator{\Fix}{Fix}
\DeclareMathOperator{\Exc}{Exc}

\DeclareMathOperator{\Sing}{Sing}

\DeclareMathOperator{\codim}{codim}
\DeclareMathOperator{\cone}{cone}
\DeclareMathOperator{\diag}{diag}

\DeclareMathOperator{\kod}{kod}

\DeclareMathOperator{\vol}{vol}

\newcommand{\Z}{\mathbb{Z}}
\DeclareMathOperator{\e}{e}
\newcommand{\OO}{\mathcal{O}}

\newcommand\dual{\mathrel{\raise3pt\hbox{$\underline{\mathrm{\thinspace d
\thinspace}}$}}}
\newcommand\qe{\ifhmode\unskip\nobreak\fi\quad $\Box$}       % box for QED

\def\BOX{\hfill\lower.5\baselineskip\hbox{$\Box$}}

\newtheorem{theorem}{Theorem}[section]
\newtheorem{lemma}[theorem]{Lemma}

\newtheorem{proposition}[theorem]{Proposition}
\newtheorem{prop}[theorem]{Proposition}

\newtheorem*{theorem*}{Theorem}
\newtheorem*{problem*}{Problem}
\newtheorem*{question*}{Question}

\theoremstyle{remark}
\newtheorem{remark}[theorem]{Remark}

\theoremstyle{definition}
\newtheorem{definition}[theorem]{Definition}

\newtheorem*{notation*}{Notation}

\numberwithin{equation}{section}

\newcounter{nootje}
\renewcommand\check[1]
  {\marginpar{\tiny\begin{minipage}{20mm}\begin{flushleft}\thenootje : #1\end{flushleft}\end{minipage}}\addtocounter{nootje}{1}}
\begin{document}

\title[Pluricanonical systems]{The pluricanonical systems of a product-quotient variety}

\author{Filippo F. Favale}
\address{Dipartimento di Matematica e Applicazioni,
	Universit\`a degli Studi di Milano-Bicocca,
	Via Roberto Cozzi, 55,
	I-20125 Milano, Italy}
\email{filippo.favale@unimib.it}

\author{Christian Gleissner}
\address{Lehrstuhl Mathematik VIII, Universit\"at Bayreuth, 
Universit\"atsstra\ss e 30, D-95447 Bayreuth, Germany.}
\email{Christian.Gleissner@uni-bayreuth.de}

\author{Roberto Pignatelli}
\address{Dipartimento di Matematica,
	Universit\`a di Trento,
	via Sommarive 14,
	I-38123 Trento, Italy.}
\email{Roberto.Pignatelli@unitn.it}

\date{\today}
\thanks{
C. Gleissner wants to thank I. Bauer for several useful discussions on this subject. \\
R. Pignatelli is indebted to G. Occhetta for his help with the proof of Proposition \ref{plurigenera==>CY}.\\
All of the authors want to thank S. Coughlan  for helpful comments.
This research started at the Department of Mathematics of the University of Trento in 2017 from a question of C. Fontanari, that we thank heartily, when the first two authors were Post-Docs there and were supported by FIRB 2012 "Moduli spaces and Applications".\\
R. Pignatelli is grateful to F. Catanese for inviting him to Bayreuth with the ERC-2013-Advanced Grant-340258-TADMICAMT; part of this research took place during his visit. He is partially supported by the project PRIN 2015 Geometria delle variet\`a algebriche.\\
F.F. Favale and R. Pignatelli are members of GNSAGA-INdAM.\\
}

 \makeatletter
    \def\@pnumwidth{2em}
  \makeatother

\begin{abstract}
We give a method for the computation of the plurigenera of a product-quotient manifold, and two different types of applications of it: to the construction of Calabi-Yau threefolds and to the determination of the minimal model of a product-quotient surface of general type.
\end{abstract}

\maketitle

\addtocontents{toc}{\protect\setcounter{tocdepth}{1}}

\tableofcontents
%=========================================================================

%%%%%%%%%%%%%%%%%%%%%%%%%%%%%%%%%%%%%%%%%%%%%%%%%%%%%%%%%%%%%%%%%%%%%%%
\section*{Introduction}

Product-quotient varieties are varieties obtained by taking a minimal resolution 
of the singularities of a quotient $X:=\left( \prod_1^n C_i \right)  /G$, the {\it quotient model}, where $G$ is a finite group acting diagonally, {\it i.e.} as $g(x_1,\ldots,x_n)=(gx_1,\ldots,gx_n)$. Usually the genera of the curves $C_i$ are assumed to be at least $2$: for the sake of simplicity, we will assume this implicitly from now on.

The notion of product-quotient variety has been introduced in \cite{BaPi12} in the first nontrivial case $n=2$, as a generalization of the {\it  varieties isogenous to a product of unmixed type}, where the action of the group is assumed to be free. 

Product-quotient varieties have proved in the last decade to form a very interesting class, because even if they are relatively easy to construct,  there are several objects with interesting properties among them. Indeed they have been a fruitful source of examples with applications in different areas of algebraic geometry.

For example \cite{AliceMatteo} constructs in this way several $K3$ surfaces with automorphisms of prime order that are not symplectic. A completely different application is the construction of rigid not infinitesimally rigid compact complex manifolds obtained in \cite{notinfinitesimally}, answering a question about 50 years old.

A classical problem is the analysis of the possible behaviors of the canonical map of a surface of general type. \cite{Beauville} provides upper bounds for both the degree of the map and the degree of its image, but very few examples realizing values near those bounds are in literature. The current best values have been recently attained respectively in \cite{ChristianRobertoCarlos} and \cite{Cat18} with this technique.

%For what concerns the classification of the surfaces of general type \cite{MichaelMatteo} study the asymptotic behavior  of the number of connected components of the Gieseker moduli space of the surfaces of general type with fixed deformation invariant $K^2$. It has been proved in \cite{Cat92} that this number is bounded above by $(K^2)^{77K^2}$; \cite{MichaelMatteo} constructs examples that show that this bound cannot be improved much in the sense that the number of components of the examples in  \cite{MichaelMatteo} grows asymptotically as $C (K^2)^{\sqrt{K^2}}$ ($C$ constant).

Last but not least, product-quotient surfaces  have been used to construct several new examples of surfaces of general type $S$ with $\chi({\mathcal O}_S)=1$, the minimal possible value, see \cite{mysurvey} and the references therein.
Restricting for the sake of simplicity to the regular case, the minimal surfaces of general type with geometric genus $p_g=0$, whose classification is a long standing problem known as {\it Mumford's  dream}, we now have dozens of families of them constructed as product-quotient surfaces, see \cites{fano,isogenous,BaCaGrPi12,BaPi12,BaPi16}, a huge number when compared with the examples constructed by other techniques, see \cite{surveypg0}. In higher dimensions, a complete classification of threefolds isogenous to a product with $\chi({\mathcal O}_X)=-1$, the maximal possible value, has been achieved recently see \cite{FG16, Christian}. 

It is very likely that  the list of product-quotient surfaces of general type with $p_g=0$ in \cite{BaPi16} is complete, but we are not able to prove it. The main obstruction to get a full classification is that it is very difficult to determine the minimal model of a regular product-quotient variety. Indeed the list was produced by a computer program able to classify all regular product-quotient surfaces $S$ with $p_g=0$ and a given value of $K^2$. The surfaces of general type $S$ with $p_g=0$ have, by standard inequalities, $1 \leq K^2_S \leq 9$ when minimal, but a minimal resolution of the singularities of a product-quotient surface may be not minimal and then have $K_S^2\leq 0$.  Detecting the rational curves with self-intersection $-1$ in one of these surfaces may be very difficult, see for example the {\it fake Godeaux surface} in \cite{BaPi12}*{Section 5}.

More generally, in birational geometry one would like to know, given an algebraic variety, one of the``simplest'' variety in its birational class, a ``minimal'' one. This is the famous Minimal Model Program, producing a variety with nef canonical system and at worse terminal singularities, or a Mori fiber space. At the moment we are not able to run a minimal model program explicitly for a general product-quotient variety even in dimension $2$. Anyhow, knowing all plurigenera $h^0(dK)$ of an algebraic variety gives a lot of information on its minimal models. 

Actually the main result of this paper is a method for computing all plurigenera of a product-quotient variety. We first prove the following
\begin{theorem*}
Let $Y$ be a smooth quasi-projective variety, let $G$ be a finite subgroup of $\Aut(Y)$ and let
$\psi \colon \widehat{X} \to Y/G=:X$ be a resolution of the singularities.
Then there exists a normal variety 
$\widetilde{Y}$, a proper birational morphism  $\phi \colon \widetilde{Y} \to Y$ and a finite surjective morphism 
$\epsilon \colon \widetilde{Y} \to \widehat{X}$ such  that the following diagram commutes:
\[
\begin{xy}
  \xymatrix{
     \widetilde{Y}\ar[rr]^{\epsilon}\ar[d]_{\phi}  &  & \widehat{X} \ar[d]^{\psi} \\   
		 Y \ar[rr]_{\pi}  &   & X
		}
\end{xy}
\]

Setting $R:=K_Y-\pi^*K_X$ and $E:=K_{\hat{X}}-\psi^*K_X$ there is a natural isomorphism  
\[
H^0\big(\widehat{X}, \sO_{\widehat{X}}(d K_{\widehat{X}})\big) \simeq 
H^0\left(Y,\sO_{Y}(dK_Y) \otimes \mathcal{I}_d \right)^G
\]
for all $d \geq 1$, where 
$\sI_d$ is the sheaf of ideals $\sO_Y(- dR) \otimes \phi_{\ast}\sO_{\widetilde{Y}} (\epsilon^{\ast} dE)$. 
\end{theorem*}
and then we show how to compute $\phi_{\ast}\sO_{\widetilde{Y}} (\epsilon^{\ast} dE)$ when $X$ has only isolated cyclic quotient singularities. 
It should be mentioned that we need an explicit   basis of $H^0\left(Y,\sO_{Y}(dK_Y)\right)$ to use the theorem. 
So in our applications we will work with equations defining $Y$.

The second motivation for this paper was to investigate methods to construct Calabi-Yau threefolds systematically. Indeed, most of the known Calabi-Yau threefolds are constructed by taking the resolution of a generic anticanonical section of a toric Fano fourfold. This idea stems from Batyrev's seminal paper \cite{Bat} but the complete list of this topologically distinct Calabi-Yau threefolds which one can obtain with this method was obtained with the help of the computer (see \cite{Kreuzer}) with the classification of the 473.800.776 reflexive polytopes in dimension $4$. Apart from these Calabi-Yau threefolds, very few examples are known and their construction involves ad hoc methods such as quotients by group actions (see, for example, \cite{GroupsActing, NewExample, CloserLook}). 

Hence, the idea of using the well-known machinery of the product-quotient varieties could prove to be effective in finding new examples of Calabi-Yau threefolds.
We first prove that no product-quotient variety can be Calabi-Yau. Still, as in \cite{AliceMatteo} for dimension $2$, they may be birational to a Calabi-Yau. We then introduce the concept of a numerical Calabi-Yau variety, that is a variety whose Hodge numbers are compatible with a possible Calabi-Yau minimal model. Then we show that a numerical Calabi-Yau product-quotient threefold is birational to a Calabi-Yau threefold if and only if all its plurigenera are equal to $1$. We construct $12$ families of numerical Calabi-Yau threefolds as product-quotient variety and use our above mentioned Theorem to compute, for two of them, their plurigenera and then determine if they are birational to a Calabi-Yau threefold or not.

Finally we apply our method to show the minimality of several product-quotient surfaces whose quotient model has several noncanonical singularities, thus disproving a conjecture of I. Bauer and the third author, namely \cite{BaPi16}*{Conjecture 1.5}.

The paper is organized as follows.

The first two sections are devoted to some possible applications of a formula for the plurigenera of product-quotient manifolds.

In section \ref{minimal models} we discuss conditions for a product-quotient variety to be minimal. Then we concentrate in the case of dimension $2$,  giving an explicit formula for the number of curves contracted by the morphism onto the minimal model in terms of the plurigenera.

In section \ref{CY 3-folds} we move to dimension $3$, discussing the product-quotient threefolds birational to Calabi-Yau threefolds.

In section \ref{list} we produce, with the help of the computer program MAGMA, $12$ families of numerical Calabi-Yau threefolds.

In section \ref{Main} we prove our main Theorem above, in Proposition \ref{UniqueSim} and Theorem \ref{mainthm}.
%, that reduces the computation of the plurigenus $P_d(\hat{X})$ of a minimal resolution of singularities of a quotient $X:=Y/G$, $Y$ smooth, to the computation of a certain ideal sheaf ${\sI}_d$ on $Y$. Then, 

In section \ref{CyclicQuotients}, we show how to compute ${\sI}_d$ when all stabilizers are cyclic, as in the case of product-quotient varieties.

In section \ref{trueCY} and \ref{fakeCY} we apply our theorem to two of the numerical Calabi-Yau threefolds produced in section \ref{list}, showing that one is birational to a Calabi-Yau threefold and the other is not. 

Finally, in section \ref{MinSurfs}, we discuss the mentioned application of our theorem to certain product-quotient surfaces and explain why this application would be difficult to achieve with existing techniques.

\begin{notation*}
All algebraic varieties in this article are complex, quasi-projective and integral, so irreducible and reduced. 

A curve is an algebraic variety of dimension $1$, a surface is an algebraic variety of dimension $2$.  

For every projective algebraic variety $X$ we consider the dimensions $q_i(X):=h^{i}(X,\sO_X)$ of the cohomology groups of its structure sheaf  for all $1 \leq i \leq \dim X$. For $i=n$ this is the {\it geometric genus} $p_g( X):=q_n(X)$, for $i<n$ they are called  {\it irregularities}. If $X$ is smooth, by Hodge Theory $q_i(X)=h^{i,0}(X):=h^0(X,\Omega^i_X)$. If $X$ is a curve, there are no  irregularities and the geometric genus is the usual genus $g(X)$.
 If $S$ is a surface the unique irregularity $q_1(S)$ is usually denoted by $q(S)$. 
 
  A normal variety $X$ is Gorenstein if its dualizing sheaf $\omega_X$ (\cite{Hartshorne}*{III.7}) is a line bundle.
 If $X$ is Gorenstein in codimension $1$ then $\omega_X$ is a Weil divisorial sheaf and we denote by $K_X$ a canonical divisor, so $\omega_X \cong \sO_X(K_X)$.  
We then define its $d$-th plurigenus $P_d(X)= h^{0}(X,\sO_X(dK_X))$. By Serre duality $P_1(X)=p_g(X)$.
$X$ is $\QQ$-Gorenstein if $K_X$ is $\QQ$-Cartier {\it i.e.} if there exists $d \in \NN$ such that $dK_X$ is Cartier. A normal variety is factorial, resp. $\QQ$-factorial if every integral Weil divisor is Cartier, resp. $\QQ$-Cartier.
 
 We use the symbols  $\sim_{lin}$ for linear equivalence of Cartier divisors,  $\sim_{num}$  for numerical equivalence of $\QQ$-Cartier divisors.

We write $\ZZ_m$ for the cyclic group of order $m$, ${\mathcal D}_m$ for the dihedral group of order $2m$, ${ \mathfrak S}_m$ for the symmetric group in $m$ letters.

For $a,b,c \in \ZZ$, $a\equiv_b c$ if nd only if $b$ divides $a-c$.
\end{notation*}

\section{Minimal models of quotients of product of two curves}\label{minimal models}

Consider a product $\prod_{i=1}^nC_i$ of smooth curves. If the genus of each curve is at least $2$, then $K_{\prod C_i}$ is ample. Moreover, if $G$ is a finite group acting freely in codimension $1$ on $\prod_{i=1}^nC_i$, as in the case of product quotient varieties (of dimension at least $2$), $K_X$ is ample too, where we have set $X:=\left( \prod_{i=1}^n C_i\right) /G$.

In particular, if $G$ acts freely then $X$ is smooth and $K_{X}$ is ample, so $X$ is a smooth minimal variety of general type. 

If the action of $G$ is free in codimension $1$ and $X$ has at worst canonical singularities, then we can take a {\it terminalization} of $X$, i.e. a crepant resolution  $\widehat{X} \rightarrow X$ of the canonical singularities of $X$ such that $\widehat{X}$ has terminal singularities. Then $K_{\widehat{X}}$ is automatically  nef and therefore $\widehat{X}$ is a minimal model of $X$. 

 If $q(\widehat{X})\neq 0$, the Albanese morphism of $\widehat{X}$ gives some obstructions to the existence of $K_{\widehat{X}}$-negative curves, since it contracts every rational curve.
Indeed the first example of a quotient $X=\left( \prod_{i=1}^n C_i\right) /G$ of general type such that a minimal resolution of the singularities  $\widehat{X}$ of $X$ is not a  minimal variety is the product-quotient surface studied in \cite{mistrettapolizzi}*{6.1}. 
%The minimal model is determined by studying the Albanese morphism of $\widehat{X}$. Indeed, in this case $q(\widehat{X})=1$, so the Albanese morphism of $\widehat{X}$ is a fibration onto an elliptic curve,  and therefore all rational curves, including all curves contracted on the minimal model, are contained in its fibres.

We find worth mentioning here that in the similar case of {\it mixed quotients}, {\it i.e.} for minimal resolutions $S$ of  singularities of a quotient $C \times C/G$ where $G$ exchange the factors, there are results guaranteeing the minimality of $S$ if $S$ is irregular (and some more assumptions, see  \cite{FurtherQuotients}*{Theorem 3} and \cite{MQEwAS}*{Theorem 4.5} for the exact statements).

% Indeed, in dimension $2$, if $G$ contains an automorphism that exchanges the factors $C_i$ (in particular $C_1 \cong C_2$)  we have the following results:
%\begin{theorem}[ \cite{FurtherQuotients}*{Theorem 3}]\label{qgeq3} Let $C$ be a curve of genus at least two and let $G$ be a group of automorphisms of $C\times C$  that  contains an automorphism exchanging the factors. 

%Let $\widehat{X}$ be the minimal resolution of the singularities of the quotient $\left( C \times C \right)/G$. If  $q(\widehat{X})\geq 3$ then $\widehat{X}$ is minimal.
%\end{theorem}
%\begin{theorem}[ \cite{MQEwAS}*{Theorem 4.5}]
%Let $C$ be a curve of genus at least two and let $G$ be a group of automorphisms of $C\times C$  that  contains an automorphism exchanging the factors.  Assume moreover that $G$ acts freely in codimension $1$. Let $\widehat{X}$ be the minimal resolution of  singularities of the quotient $\left( C \times C \right)/G$. If  $q(\widehat{X})\neq 0$, then $\widehat{X}$ is minimal.
%\end{theorem}

%Both results are sharp.

If $q(\widehat{X})=0$ we have no Albanese morphism and then determining the minimal model is much more difficult. The first example in the literature of a  product-quotient variety $\widehat{X}$ that is not minimal with  $q(\widehat{X})=0$ is the {\it fake Godeaux surface} in \cite{BaPi12}*{Section 5}, whose minimal model is determined by a complicated {\it ad hoc} argument. 

 See also \cite{BaPi16}*{Section 6} for some conjectures and partial results about sufficient conditions for the minimality of $\widehat{X}$ when $q(\widehat{X})=0$.

On the other hand, a lot of information on the birational class of $X$ can be obtained without running an explicit minimal model program for it, by computing some of the birational invariants of $X$.
The geometric genus and the irregularities of $\widehat{X}$ are its simplest birational invariants. They are not difficult to compute for product-quotient varieties.
The next natural birational  invariants to consider are the plurigenera $P_d(\widehat{X})$, $\forall d \in \NN$. They determine a very important birational invariant, the Kodaira dimension $\kod(\widehat{X})$. 
%This  is  the fundamental invariant used by the Enriques-Kodaira classification of surfaces and its higher dimensional analogues.
If $X$ is of general type, {\it i.e.} $\kod(\widehat{X})=\dim(\widehat{X})$, an important role in the classification theory is played by the volume 
\[
\vol (K_{\widehat{X}}):=(\dim X)! \limsup_{m\rightarrow \infty}\frac{P_m(X)}{m^{\dim X}}
\]
of its canonical divisor, that is a birational invariant determined by the plurigenera. 

Indeed, let us now restrict for the sake of simplicity to the case $n=2$. If $\widehat{X}$ is a surface of general type, then it is well known that it has a unique minimal model $X_{min}$. The natural map of $\widehat{X}$ on its minimal model is the composition of $r$ elementary contractions, where $r=  \vol (K_{\widehat{X}}) - K^2_{\widehat{X}}$, and  $\vol (K_{\widehat{X}})$ equals the self intersection of a canonical divisor of the minimal model.

By \cite{BPV}*{Proposition 5.3} a surface of general type $S$ is minimal if and only if $h^1(\sO_{S}(dK_S))=0$ for all $d \geq 2$.
Then, by Riemann-Roch, $P_d(S)=\chi(\sO_S)+\binom{d}{2}K^2_S$, and therefore
\begin{equation}\label{K2min}
\binom{d}{2}  K^2_S=P_d(S)+q(S)-p_g(S)-1.
\end{equation} 
Since the right-hand side of (\ref{K2min}) is a birational invariant it follows that if $\widehat{X}$ is of general type, then 
\[
\vol (K_{\widehat{X}})=\frac{P_3(\widehat{X})-P_2(\widehat{X})}2=P_2(\widehat{X})+q(\widehat{X})-p_g(\widehat{X})-1.
\]

By the Enriques-Kodaira classification and Castelnuovo rationality criterion, every surface $\widehat{X}$ with  $K^2_{\widehat{X}}>0$ and $P_2(\widehat{X}) \neq 0$  is of general type, so we have the following well known proposition: 
\begin{proposition}\label{vol}
Assume  $\widehat{X}$  is a surface with $K^2_{\widehat{X}}>0$ and $P_2(\widehat{X}) \neq 0$.

Then $\widehat{X}$  is a surface of general type and
\[
\vol (K_{\widehat{X}})
=
P_2(\widehat{X})+q(\widehat{X})-p_g(\widehat{X})-1=
\frac{P_3(\widehat{X})-P_2(\widehat{X})}2.
\]
\end{proposition}

Similarly, we can compute the volume of the canonical divisor of $\hat{X}$ if we know any pair of plurigenera  $P_d$, $d \geq 2$, or one of its plurigenera, geometric genus and all irregularities. Once we compute $K_X^2$, an easy computation, we immediately deduce whether $\hat{X}$ is minimal and more generally the number $r$ of irreducible curves of $\hat{X}$ contracted on the minimal model.

\section{Product quotient varieties birational to Calabi-Yau threefolds}\label{CY 3-folds}

%An important class of varieties is the class of the Calabi-Yau varieties. The smooth Calabi-Yau varieties are one of the three building blocks of the Beauville-Bogomolov decomposition \cite{BeBo} of smooth projective varieties (and more generally, of compact K\"ahler manifolds) with trivial first Chern class.
The Beauville-Bogomolov theorem has been recently extended to the singular case \cite{AndreasThomas}, requiring an extension of the notion of Calabi-Yau to minimal models. The following is the natural definition, a bit more general than the one necessary for the Beauville-Bogomolov decomposition in \cite{AndreasThomas}.

\begin{definition}
A complex projective variety $Z$ with at most terminal singularities is called Calabi-Yau if it is Gorenstein,
\[
K_Z \sim_{lin} 0 \qquad \makebox{and} \qquad q_i(Z)=0 \qquad \forall 1 \leq i \leq \dim Z -1. 
\]
\end{definition}

Calabi-Yau varieties of dimension $2$ are usually called $K3$ surfaces. 

We first show that there is no Calabi-Yau product-quotient variety.

\begin{proposition}  
Let  $X=(C_1 \times \ldots \times C_n)/G$ be 
the quotient model of 
a product-quotient variety and let $\rho \colon \widehat{X} \to X$ be a partial resolution of the singularities of $X$ such that $\widehat{X}$ has at most terminal singularities. 

Then $K_{\widehat{X}} \not\sim_{num} 0$.
\end{proposition}

\begin{proof}
Let $\pi \colon \prod C_i \rightarrow X$ be the quotient map. Then $\pi$ is unramified in codimension $1$. Since $K_{\prod C_i}$ is ample, then $K_X$ is ample too, so it has strictly positive intersection with every curve of $X$. Since $\codim \Sing X \geq 2$ one can easily find a curve $C$ in $X$ not containing any singular point of $X$: for example a general fibre of the projection   $X=(C_1 \times \ldots \times C_n)/G \rightarrow (C_2 \times \ldots \times C_n)/G$. Set $\widehat{C}=\rho^*C$. Then $K_{\widehat{X}}\widehat{C}=K_XC\neq 0$ and therefore $K_{\widehat{X}} \not\sim_{num} 0$.
\end{proof}

So there is no hope to construct a Calabi-Yau variety directly as partial resolution of the singularities of a product-quotient variety, but one can still hope to get something birational to a Calabi-Yau variety. \cite{AliceMatteo} constructed several $K3$ surfaces that are birational to product-quotient varieties. Their method starts by constructing  product-quotient surfaces with  $p_g=1$ and $q=0$. 

We follow a similar approach for constructing Calabi-Yau threefolds. This leads to the following definition:
\begin{definition}
A normal threefold $\widehat{X}$ is a numerical Calabi-Yau if 
\[
p_g(\widehat{X})=1, \qquad q_i(\widehat{X})=0 \qquad 
\makebox{for} \qquad i=1,2. 
\]
\end{definition}

\begin{proposition}
Let $\widehat{X}$ be a product-quotient threefold.
Assume that $\widehat{X}$ is birational to a Calabi-Yau threefold. Then $\widehat{X}$ is a numerical Calabi-Yau threefold.
\end{proposition}

\begin{proof} Let $Z$ be a Calabi-Yau threefold birational to $\widehat{X}$. To prove that $\widehat{X}$ is a numerical Calabi-Yau, we take a common resolution $\widehat{Z}$ of the singularities of 
$Z$ and of $\widehat{X}$. Since $\widehat{X}$ and $Z$ have terminal singularities, and terminal singularities are rational (see \cite{Elkik}), it follows that
\[
p_g(Z)=p_g(\widehat{Z})=p_g(\widehat{X}) \quad \makebox{and} \quad q_i(Z)=q_i(\widehat{Z})=q_i(\widehat{X}) 
\]
by the Leray spectral sequence. 
\end{proof}

\begin{remark}
It follows that the quotient model of a numerical Calabi-Yau product-quotient threefold has at least one singular point that is not canonical.
Indeed, since the quotient map $\pi \colon C_1 \times C_2 \times C_3 \to X$ is quasi-\'etale and the curves $C_i$ have genus at least two, then $K_X$ is ample.
 If $X$ had only canonical singularities, then $\widehat{X}$ would be of general type, and so would be $Z$, a contradiction.
\end{remark}
\begin{remark}
Let $X$ be the quotient model of a numerical Calabi-Yau product-quotient threefold and 
$\rho \colon \widehat{X} \to X$ be a resolution, then $p_g(\widehat{X})=1 \Rightarrow \kappa(\widehat{X})\neq -\infty$.
Now we run a Minimal Model Program on $\widehat{X}$. Assume that it ends with a Mori fibre space, then 
$\kappa(\widehat{X})=-\infty$ according to \cite{mat}*{Theorem 3-2-3}) which is impossible. Therefore, the Minimal Model Program ends with a threefold $Z$ with terminal singularities and $K_{Z}$ nef. 
\end{remark}

We close this section with its main result,  a criterion to decide whether a numerical Calabi-Yau product-quotient 
threefold is birational to a Calabi-Yau threefold.

\begin{proposition}\label{plurigenera==>CY}
Let $\widehat{X}$ be a numerical Calabi-Yau product-quotient threefold.

If $P_d(\widehat{X})=1$ for all $d \geq 1$, then  $\widehat{X}$ is birational to a Calabi-Yau threefold.
\end{proposition}

\begin{proof} Let $Z$ be a minimal model of $\widehat{X}$. It suffices to show that $K_Z$ is trivial. 
According to Kawamata's abundance for minimal threefolds \cite{Kawa}, some multiple $m_0K_Z$ is base point free. By assumption $h^0(m_0K_Z)=h^0(m_0K_{\widehat{X}})= 1$, which implies that 
$m_0K_Z$ is trivial. In particular $m_0K_{Z^0}$ is trivial, where $Z^0=Z \setminus Sing(Z)$ is the smooth locus. Since $Z$ has terminal singularities $h^0(K_{Z^0})=h^0(K_{\widehat{X}})=1$ and it follows that $K_{Z^0}$ is trivial. By normality, $K_Z$ must be also trivial. 
\end{proof}

\section{Examples of numerical Calabi-Yau product-quotient threefolds}
\label{list}

In this section we present an algorithm that allows us to systematically search for numerical 
Calabi-Yau threefolds. We use a 
MAGMA implementation of this algorithm to produce a list of examples of such threefolds. 
For a detailed account about classification algorithms and the 
language of product quotients, we refer to \cite{ChristianThesis}.

\noindent
To describe the idea of the algorithm, suppose that the quotient model of  a numerical Calabi-Yau threefold
\[
X=\big(C_1 \times C_2 \times C_3\big)/G 
\]
is given. Then  $C_i/G \cong \PP^1$ and we have three  $G$-covers $f_i \colon C_i \to \mathbb P^1$. 
Let $b_{i,1}, \ldots, b_{i,r_i}$ be the branch points of $f_i$ and denote by 
$T_i:=[m_{i,1}, \ldots, m_{i,r_i}]$ the three unordered  lists of branching indices, these will be called the types in the sequel.  

\begin{proposition}
The type $T_i:=[m_{i,1}, \ldots, m_{i,r_i}]$ satisfies the following properties:
\begin{itemize}
\item[i)] $ m_{i,j} \leq  4g(C_i) + 2$,
\item[ii)] $m_{i,j}$ divides the order of $G$, 
\item[iii)] $r_i \leq \dfrac{4\big(g(C_i) - 1 \big)}{n} + 4$,
\item[iv)] ${\displaystyle 2g(C_i)-2= |G| \bigg(-2 + \sum_{j=1}^{r_i} \frac{m_{i,j}-1}{m_{i,j}}\bigg)}$ 
\end{itemize}
\end{proposition}

\begin{proof}
$ii)$ follows from the fact that the $m_{i,j}$ are the orders of the stabilizers of the points above the branch 
points $b_{i,j}$.  

$i)$ is an immediate consequence of the classical bound of Wiman \cite{Wiman} for the order of an automorphism of a curve of genus at least $2$, since the stabilizers are cyclic.

$iv)$ is the Riemann-Hurwitz formula. $iii)$ follows from $iv)$ and  $m_{i,j} \geq 2$. 
\end{proof}

{\bf 1st Step:} 
The first step of the algorithm is based on the proposition above. 
As an input value we fix an integer $g_{max}$. The output is a full list of numerical Calabi-Yau product-quotient threefolds, such that the genera of the curves $C_i$ are bounded from above by $g_{max}$.

According to the Hurwitz bound on the automorphism group, we have
\[
|Aut(C_i)| \leq 84(g_{max} -1).
\]
Consequently there are only finitely many possibilities for the order $n$ of the group $G$. 
On the other hand, for fixed $g_i \leq g_{max}$ and fixed group order $n$, there are only finitely 
many possibilities for integers  $m_{i,j} \geq 2$ fulfilling the constraints from the proposition above. 
We wrote a MAGMA code, that returns all admissible combinations 
\[
[g_1,g_2,g_3,n,T_1,T_2,T_3]. 
\]

\bigskip
\noindent
{\bf 2nd Step:} For each tuple $[g_1,g_2,g_3,n,T_1,T_2,T_3]$
determined in the first step, we search through the groups $G$ of order $n$ and check if we can 
realize three $G$ covers $f_i \colon C_i \to \mathbb P^1$
with branching indices $T_i:=[m_{i,1}, \ldots, m_{i,r_i}]$. 
By \emph{Riemann's existence theorem} such  covers exist if and only if 
there are elements $h_{i,j} \in G$ of order $m_{i,j}$, which generate $G$ and fulfill the relations 
\[
\prod_{j=1}^{r_i}  h_{i,j} = 1_G \qquad \makebox{for each} \quad 1 \le i \le 3. 
\]
Let $X$ be the quotient of $C_1 \times C_2 \times C_3$ by the diagonal action of $G$. 
The singularities 
\[
\frac{1}{n}(1,a,b)
\]
of  $X$ can be determined  using the elements $h_{i,j}$  
cf. \cite{BaPi12}*{Proposition 1.17}. The same is true for the invariants $p_g$ and $q_i$ of a resolution cf. 
\cite{FG16}*{Section 3}, since they are given as the dimensions of the $G$-invariant parts of 
$H^0(\Omega^p_{C_1 \times C_2 \times C_3})$, 
which can be determined using the formula formula of \emph{Chevalley-Weil} see 
\cite{FG16}*{Theorem 2.8}. 
The threefolds with only canonical singularities are discarded as well as those 
with invariants different from $p_g=1$, $q_1=q_2=0$. 
As an output we return the following data of $X$: 
the group $G$, the types $T_i$, the set of canonical singularities $\mathcal S_c$ and the set of 
non-canonical singularities $\mathcal S_{nc}$. 

We run our MAGMA implementation of the algorithm for $g_{max}=6$ and the additional restriction that the 
$f_i \colon C_i \to \mathbb P^1$
are branched in only three points i.e. $r_i=3$. The output is in Table \ref{figure CY}.

They may be birational to a Calabi-Yau threefold or not. Both cases occur, as we will see in Sections \ref{trueCY} and \ref{fakeCY}.

\begin{table}{H} 
{\scriptsize 
\begin{tabular}{|l|l|l|l|l|l|l|l|}
\hline
No. & $G$ & $Id$ &$T_1$ & $T_2$ & $T_3$  &  $\mathcal S_c$ & 
$\mathcal S_{nc}$ \\ 
\hline
\hline
1 & $\mathbb Z_6$ & $\langle 6,2 \rangle$ &$[3,6,6] $ & $[3,6,6] $ & $[3,6,6] $  &  
$\frac{(1,1)}{3}^4, \frac{(2,2)}{3}^{24}$ & 
$\frac{(1,1)}{6}^8$ \\ 
\hline
2 & $\mathbb Z_8$ & $\langle 8, 1\rangle$ &$[2,8,8]$ & $[2,8,8]$ & $[4,8,8]$  &  
$\frac{(1,1)}{2}^{32}, \frac{(1,3)}{4}^3$ & 
$\frac{(1,1)}{4},  \frac{(1,1)}{8}^2,  \frac{(1,3)}{8}^6$ \\ 
\hline
3 & $\mathbb Z_{10}$ & $\langle 10, 2 \rangle$ &$[2,5,10] $ & $[2,5,10] $ & $[5,10,10]$  &  
$\frac{(1,1)}{2}^{14},  \frac{(1,4)}{5},  \frac{(2,2)}{5}^8, 
\frac{(2,3)}{5}^4$ & 
$\frac{(1,2)}{5}^4,  \frac{(1,1)}{10}^2$ \\ 
\hline 
4 & $\mathbb Z_{12}$ & $\langle 12, 2 \rangle$ &$[2,12,12]$ & $[2,12,12]$ & $[3,4,12]$  &  
$\frac{(1,1)}{2}^{40},  \frac{(1,1)}{3},  \frac{(2,2)}{3}^3$ & 
$\frac{(1,1)}{4}^4,  \frac{(1,1)}{12},  \frac{(5,5)}{12}^3$ \\ 
\hline
5 & $(\mathbb Z_4 \times \mathbb Z_4) \rtimes \mathbb Z_2$ & $\langle 32, 11 \rangle$ & 
$[2,4,8]$ & $[2,4,8]$ & $[2,4,8]$  &  
$\frac{(1,1)}{2}^{24}$ & $
\frac{(1,1)}{4}^6,  \frac{(1,1)}{8},  
\frac{(1,5)}{8}^3$ \\ 
\hline
6 & $(\mathcal D_4 \times \mathbb Z_2) \rtimes \mathbb Z_4$ & $\langle 64, 8 \rangle$ & 
$[2,4,8]$ & $[2,4,8]$ & $[2,4,8]$  &  
$\frac{(1,1)}{2}^{60}$ & $\frac{(1,1)}{4}^{6}, \frac{(1,1)}{8}^{4} $ \\ 
\hline
7 & $\mathfrak S_4 \times \mathbb Z_3$ & $\langle 72, 42 \rangle$ & $[2,3,12]$ & $[2,3,12]$ & $[2,3,12]$  &  
$\frac{(1,1)}{2}^{36}, \frac{(1,1)}{3}^{17}$ & $\frac{(1,1)}{12}, \frac{(1,7)}{12}^{3} $ \\ 
\hline
8 & $\mathbb Z_2^4 \rtimes \mathbb Z_5$ & $\langle 80, 49 \rangle$ & 
$[2,5,5]$ & $[2,5,5]$ & $[2,5,5]$  &  
$\frac{(1,4)}{5}^{6} $ & 
$\frac{(1,1)}{5}^{2}$ \\ 
\hline
9 & $\mathbb Z_2^4 \rtimes \mathbb Z_5$ & $\langle 80, 49 \rangle$ &$[2,5,5]$ & $[2,5,5]$ & $[2,5,5]$  &  
$\frac{(1,3)}{5}^{2}, \frac{(3,4)}{5}^{4} $ & $\frac{(1,2)}{5}^{2} $ \\ 
\hline
10 & $\mathbb Z_4^2 \rtimes \mathfrak S_3$ & $\langle 96, 64 \rangle$ & $[2,3,8]$ & $[2,3,8]$ & $[2,3,8]$  &  
$\frac{(1,1)}{2}^{16}, \frac{(1,1)}{3}, \frac{(2,2)}{3}^{3}$ & 
$\frac{(1,1)}{4}^{6}, \frac{(1,1)}{8}, \frac{(1,5)}{8}^{3} $ \\ 
\hline
11 & $GL(3,\mathbb F_2)$ & $\langle 168, 42 \rangle$ &$[2,3,7] $ & $[2,3,7] $ & $[2,3,7]$  &  
$\frac{(1,1)}{2}^{16},\frac{(1,1)}{3}, \frac{(2,2)}{3}^3, \frac{(2,4)}{7}^2 $ 
& $\frac{(1,1)}{7}, \frac{(1,4)}{7}^{3}, \frac{(4,4)}{7}^{3}$ \\ 
\hline
12 & $G_{192}$ & $\langle 192, 181 \rangle$ & $[2,3,8]$ & $[2,3,8]$ & $[2,3,8]$  &  
$\frac{(1,1)}{2}^{28}, \frac{(1,1)}{3}^{4}, \frac{(2,2)}{3}^{12} $ 
& $\frac{(1,1)}{4}^{6}, \frac{(1,1)}{8}^{4}$ \\
\hline
\end{tabular}
\caption{\scriptsize Some numerical Calabi-Yau product-quotient threefolds.  Each row corresponds to a threefold, each column to one of the data of the construction: from left to right the group $G$, its Id in the MAGMA database of finite groups, the three types, the canonical singularities and the singularities that are not canonical.  The symbol $\frac{(a,b)}{n}^{\lambda}$ used in the last two columns of the table denotes $\lambda$ cyclic quotient singularities of type $\frac{1}{n}(1,a,b)$. We recall the definition in Section \ref{CyclicQuotients}. }\label{figure CY}
}
\end{table}

\section{The sheaves of ideals \texorpdfstring{$\sI_d$}{Id} on a smooth projective variety with a finite group action }\label{Main}

\begin{proposition}\label{UniqueSim}
Let $Y$ be a smooth quasi-projective variety, let $G$ be a finite subgroup of $\Aut(Y)$ and let
$\psi \colon \widehat{X} \to Y/G$ be a resolution of the singularities.
Then there exists a normal variety 
$\widetilde{Y}$, a proper birational morphism  $\phi \colon \widetilde{Y} \to Y$ and a finite surjective morphism 
$\epsilon \colon \widetilde{Y} \to \widehat{X}$ such  that the following diagram commutes:
\[
\begin{xy}
  \xymatrix{
     \widetilde{Y}\ar[rr]^{\epsilon}\ar[d]_{\phi}  &  & \widehat{X} \ar[d]^{\psi} \\   
		 Y \ar[rr]_{\pi}  &   & Y/G
		}
\end{xy}
\]
Up to isomorphism  $\widetilde{Y}$ is the normalisation of the fibre product 
$Y \times_{Y/G} \widehat{X}$ and $\phi$ and $\epsilon$ are the natural maps. 
\end{proposition}

The proof of this proposition is just a combination of the universal property of the fibre product and the universal property of the normalisation.

\begin{remark}
Note that $\widetilde{Y}$  in general fails to be smooth cf. \cite{Koll07}*{Example 2.30}. 
\end{remark}

\begin{proposition}\label{widetildQ}
The $G$ action on $Y$ lifts to an action on $\widetilde{Y}$ such that $\widehat{X}$ is the quotient.
\end{proposition}

\begin{proof}
Consider the natural $G$ action on $Y \times_{Y/G} \widehat{X}$. By the universal property of the normalisation 
it lifts to an action on $\widetilde{Y}$.
The birational map $\widetilde{Y}/G \to \widehat{X}$  induced by $\epsilon$ is finite and  therefore 
an isomorphism by Zariski's Main theorem. 
\end{proof}

\begin{remark}\label{RamForm}
Let $Y$ be a normal quasi-projective variety and $G < \Aut(Y)$ be a finite group, then
the quotient map $\pi \colon Y \to X:=Y/G$ induces an isomorphism 
\[
\pi^{\ast} \colon H^0(X,\sL) \simeq H^0(Y,\pi^{\ast}\sL)^G
\] 
for any line bundle $\sL$ on $X$.
The quotient $X:=Y/G$ is a normal $\mathbb Q$-factorial quasi-projective variety. In particular 
$K_X$ is $\mathbb Q$-Cartier. Let $\psi \colon \widehat{X} \to X$ be a resolution of singularities and $K_X$ be a canonical divisor, then 
\[
K_{\widehat{X}} = \psi^{\ast} K_X + E,
\]
where $E$ is a $\mathbb Q$-divisor supported on the exceptional locus $\Exc(\psi)$. 
Since $\pi$ is finite and $Sing(X) \subset X$ has codimension $\geq 2$, Hurwitz formula holds:
\[
K_Y=\pi^{\ast}K_X + R.  
\]
We point out that $Y$ is smooth, and thus the ramification divisor $R$ is a Cartier divisor. 
\end{remark}

\begin{theorem}\label{mainthm}
Under the assumptions from Proposition \ref{UniqueSim}, 
 there is a natural isomorphism  
\[
H^0\big(\widehat{X}, \sO_{\widehat{X}}(d K_{\widehat{X}})\big) \simeq 
H^0\left(Y,\sO_{Y}(dK_Y) \otimes \mathcal{I}_d \right)^G
\]
for all $d \geq 1$, where 
$\sI_d$ is the sheaf of ideals $\sO_Y(- dR) \otimes \phi_{\ast}\sO_{\widetilde{Y}} (\epsilon^{\ast} dE)$. 
\end{theorem}

\begin{remark}
If we write $E=P-N$, where $P,N$ are effective without common components, then 
$\sI_d \cong \sO_Y(- dR) \otimes \phi_{\ast}\sO_{\widehat{X}} (-\epsilon^{\ast} dN)$.
\end{remark}

\begin{proof}
Using Remark \ref{RamForm}, we compute 
\begin{align*}
\epsilon^{\ast}dK_{\widehat{X}}=& \epsilon^{\ast}(\psi^{\ast}dK_X + dE)  \\
=& \epsilon^{\ast}\psi^{\ast}dK_X + \epsilon^{\ast} dE \\
=& \epsilon^{\ast}\psi^{\ast}dK_X + \epsilon^{\ast} dE \\
=& \phi^{\ast} \pi^{\ast} dK_X + \epsilon^{\ast} dE \\
=& \phi^{\ast} (dK_Y - dR)+ \epsilon^{\ast} dE .
\end{align*}
Since the divisors  $\epsilon^{\ast}dK_{\widehat{X}}$ and $\phi^{\ast} (dK_Y - dR)$ are Cartier, the divisor 
$\epsilon^{\ast} dE$ is also Cartier and we obtain the isomorphism of line bundles 
\[
\sO_{\widetilde{Y}}(\epsilon^{\ast} dK_{\widehat{X}}) \cong
\sO_{\widetilde{Y}}(\phi^{\ast}(dK_Y-dR)) \otimes \sO_{\widetilde{Y}} (\epsilon^{\ast} dE)
\]
According to Proposition \ref{widetildQ}  $\widehat{X}$ is the quotient of $\widetilde{Y}$ by $G$. By 
Remark \ref{RamForm} 
\[
H^0\big(\widehat{X},\sO_{\widehat{X}}(dK_{\widehat{X}})\big) \simeq 
H^0\big(\widetilde{Y},\sO_{\widetilde{Y}}(\phi^{\ast}(dK_Y - dR)) \otimes\sO_{\widetilde{Y}}(\epsilon^{\ast} dE)\big)^G
\]
Using the projection formula:
\[
H^0\big(\widetilde{Y},\sO_{\widetilde{Y}}(\phi^{\ast}(dK_Y - dR)) \otimes \sO_{\widetilde{Y}}(\epsilon^{\ast} dE)\big)^G
=H^0\big(Y,\sO_{Y} (dK_Y - dR)\otimes \phi_{\ast}\sO_{\widetilde{Y}} (\epsilon^{\ast} dE)\big)^G.
\]
\end{proof}

Theorem \ref{mainthm} gives a method to compute the plurigenus $P_d(\widehat{X})$, if we can  determine 
the sheaf of ideals $\phi_{\ast}\mathcal O_{\widetilde{Y}} (\epsilon^{\ast} dE)$ and know a basis of $H^0\left(Y,\sO_{Y}(dK_Y)\right)$ explicitly. 
 In the next section we explain how  to compute these ideals, under the assumption that $X$ 
has only isolated cyclic quotient singularities. 

\section{The sheaves of ideals \texorpdfstring{$\sI_d$}{Id} for cyclic quotient singularities}\label{CyclicQuotients}

In this section we specialize to the case of a $G$-action, where the fixed locus 
of every automorphism $g \in G$ is isolated 
 and the stabilizer of each point $y \in Y$ is cyclic. 
Under this assumption, each singularity of $X=Y/G$ is an 
isolated cyclic quotient singularity 
\[
\frac{1}{m}(a_1,\ldots,a_n),
\]
i.e. locally in the analytic topology, 
around the singular point the variety $Y/G$  is isomorphic to  a quotient $\mathbb C^n/H$, where 
$H \simeq \mathbb Z_m$ is a cyclic group generated by a diagonal matrix 
\[
\diag\big(\xi^{a_1}, \ldots,\xi^{a_n}\big), \quad \makebox{where} \quad \xi:=
\exp\bigg(\frac{2\pi \sqrt{-1}}{m}\bigg) \quad \makebox{and} \quad \gcd(a_i,m)=1.
\]

\noindent 
In the sequel, we use toric geometry
to construct a resolution $\widehat{X}$ of the quotient $Y/G$ and give a local description of
the variety $\widetilde{Y}$ in Proposition \ref{UniqueSim}. 
We start by collecting some basics about cyclic quotient singularities from the toric point of view. 
For details we refer to \cite[Chapter 11]{CLS}.

\begin{remark}
\makebox{$$}
\begin{itemize}
\item
As an affine toric variety, the singularity $\frac{1}{m}(a_1,\ldots,a_n)$ is given by the lattice 
\[
N:=\mathbb Z^n + \frac{\mathbb Z}{m}(a_1,\ldots,a_n) \quad   \makebox{and the cone}  \quad 
\sigma:=\cone(e_1, \ldots, e_n),
\]
where the vectors $e_i$ are the euclidean unit vectors. We denote this affine toric variety by $U_{\sigma}$.
\item 
The inclusion $i \colon (\mathbb Z^n, \sigma) \to (N,\sigma)$ induces 
the quotient map 
\[
\pi \colon \mathbb C^n \to \mathbb C^n / \mathbb Z_m.
\]  
\item 
There exists a subdivision of the cone $\sigma$, yielding a fan $\Sigma$ such that the toric variety 
$\widehat{X_{\Sigma}}$ is smooth and the morphism $\psi \colon \widehat{X_{\Sigma}} \to U_{\sigma}$ induced by the identity map of the 
lattice $N$ is a resolution of  $U_{\sigma}$ i.e. birational and proper. 
\end{itemize}
\end{remark}

\noindent 
Now, the local construction of $\widetilde{Y}$ as a toric variety
is straightforward. Observe that 
the fan $\Sigma$ is also a fan in the lattice $\mathbb Z^n$. We define $\widetilde{Y_{\Sigma}}$ 
to be the toric variety associated to $(\mathbb Z^n, \Sigma)$. 
The commutative diagram 
\[
\begin{xy}
  \xymatrix{
     (\mathbb Z^n,\Sigma) \ar[rr] \ar[d] &  &  (N,\Sigma) \ar[d] \\   
		 (\mathbb Z^n,\sigma)  \ar[rr] &   &  (N,\sigma) 
		}
\end{xy}
\]
of inclusions induces a 
commutative diagram of toric morphisms, which is the local version of the diagram from Proposition \ref{UniqueSim}: 

\[
\begin{xy}
  \xymatrix{
     \widetilde{Y_{\Sigma}} \ar[rr]^{\epsilon}\ar[d]_{\phi} &  &  \widehat{X_{\Sigma}} \ar[d]^{\psi}  \\   
		 \mathbb C^n  \ar[rr]_{\pi} &   &  U_{\sigma} =\mathbb C^n/\mathbb Z_m
		}
\end{xy}
\]
Indeed, the following proposition holds:

\begin{proposition}
The map  $\epsilon \colon \widetilde{Y_{\Sigma}} \to \widehat{X_{\Sigma}}$ is finite and surjective and 
$\phi \colon   \widetilde{Y_{\Sigma}} \to \mathbb C^n$ is birational and proper. 
\end{proposition}

\begin{proof}
We need to show that 
$\mathbb C[N^{\vee} \cap \tau^{\vee}] \subset \mathbb C[\mathbb Z^n \cap \tau^{\vee}]$ is a finite ring extension
for all cones $\tau$ in  $\Sigma$. 
Clearly, any element of the form $c \chi^q \in \mathbb C[\mathbb Z^n \cap \tau^{\vee}]$ is integral over 
$\mathbb C[N^{\vee} \cap \tau^{\vee}]$, because 
$mq \in N^{\vee} \cap \tau^{\vee}$ and 
$c \chi^q$ 
solves the monic equation $x^m-c^m\chi^{mq}=0$. The general case follows from the fact that any element in 
$\mathbb C[\mathbb Z^n \cap \tau^{\vee}]$
is a finite sum of elements of the form $c \chi^q$ and finite sums of integral elements are also integral. 
Since $\Sigma$ is a refinement of $\sigma$ the morphism $\phi$ is birational and proper according to  
\cite{CLS}*{Theorem 3.4.11}. 
\end{proof}

\noindent 
For the next step, we describe how to determine the discrepancy divisor in $\widehat{X}$ over each singular 
point of the quotient $Y/G$ and its pullback under the morphism $\epsilon$.

\begin{proposition}[{\cite{CLS}*{Proposition 6.2.7 and Lemma 11.4.10}}]\label{discrepancies}
\makebox{$$}
\begin{itemize}
\item
The exceptional prime divisors of the birational morphisms 
\[
\psi \colon  \widehat{X_{\Sigma}} \to U_{\sigma} \quad \quad \makebox{and} \quad  \quad 
\phi \colon \widetilde{Y_{\Sigma}} \to \mathbb C^n
\] 
are in one to one correspondence with the rays $\rho \in \Sigma \setminus \sigma$. 
\item
Write 
$v_{\rho} \in N$ for the primitive generator of the ray $\rho$ and 
$E_{\rho} \subset \widehat{X_{\Sigma}}$ for the corresponding prime divisor, then 
$K_{X_{\Sigma}} = \psi^{\ast} K_{U_{\sigma}} +  E$,   where
\[
E:=\sum_{\rho \in \Sigma \setminus \sigma} 
\big(\langle v_{\rho}, e_1+ \ldots + e_n \rangle-1\big)E_{\rho}. 
\]
\item 
Write $w_{\rho} \in \mathbb Z^n$ for the primitive generator of the ray $\rho$ and 
$F_{\rho} \subset \widetilde{Y_{\Sigma}}$ for the corresponding prime divisor, then 
\[
\epsilon^{\ast}E_{\rho} = \lambda_{\rho} F_{\rho} \qquad \makebox{where} \qquad
\lambda_{\rho} > 0 \qquad \makebox{such that} \qquad
w_{\rho} =\lambda_{\rho} v_{\rho}.
\]
In particular 
\[
\epsilon^{\ast} E =
\sum_{\rho \in \Sigma \setminus \sigma} 
\big(\langle w_{\rho}, e_1+ \ldots + e_n \rangle-1\big)F_{\rho}.
\]
\end{itemize}
\end{proposition}

\noindent 
It remains to determine the pushforward $\phi_{\ast}\mathcal O_{\widetilde{Y_{\Sigma}}}(\epsilon^{\ast} d E)$ for $d \geq 1$. 
We provide a recipe to compute $\phi_{\ast}\mathcal O_{\widetilde{Y_{\Sigma}}}(\epsilon^{\ast} D)$  for a general Weil divisor $D$ supported on the exceptional locus of
$\phi$. 

\begin{proposition}\label{pushfor}
Let $\phi \colon \widetilde{Y_{\Sigma}}\to \mathbb C^n$ be the birational morphism from above and 
\[
D= \sum_{\rho \in \Sigma \setminus \sigma}  u_{\rho} F_{\rho}, \quad \quad u_{\rho} \in \mathbb Z
\]
be a Weil divisor,  supported on the exceptional locus of $\phi$. 
For each integer $k \geq 1$, we define the sheaf of ideals 
$\mathcal I_{kD}:= \phi_{\ast} \mathcal O(kD)$, then:

\bigskip 
\begin{itemize}
\item[i)]
The ideal of global sections  $I_{kD} \subset \mathbb C[x_1, \ldots , x_n]$ is given by  
\[
I_{kD}=\bigoplus_{\alpha \in k P_D \cap \mathbb Z^n} \mathbb C \cdot \chi^{\alpha}, 
\]
 where $
P_D:=\lbrace u \in \mathbb R^n ~ | ~ u_i \geq 0, \quad \langle u, w_{\rho} \rangle \geq - u_{\rho}\rbrace$
is the polyhedron associated to $D$.

\bigskip
\item[ii)]
Let $l=(l_1,\ldots,l_n)$ be a tuple of positive integers such that $l_i \cdot e_i \in P_D$ and define 
\[
\Box_{l}:=\lbrace y \in \mathbb R^n ~ | ~ 0 \leq y_i \leq l_i \rbrace.
\]
Then, the set of monomials $\chi^{\alpha}$, where $\alpha$ is a lattice point in 
the polytope $k(\Box_{l} \cap P_D)$
generate $I_{kD}$. 
\end{itemize}
\end{proposition}

\begin{proof}
i) By definition of the pushforward and the surjectivity of $\phi$, we have
\[
I_{kD}= \phi_{\ast} \sO_{\widetilde{Y_{\Sigma}}}(kD)(\mathbb C^n) = H^0(\widetilde{Y_{\Sigma}},\sO_{\widetilde{Y_{\Sigma}}}(kD)).
\]
According to \cite{CLS}*{Proposition 4.3.3}, it holds
\[
H^0(\widetilde{Y_{\Sigma}},\mathcal O(kD))= 
\bigoplus_{\alpha \in P_{kD} \cap \mathbb Z^n} \mathbb C \cdot \chi^{\alpha}
\]
and the claim follows since $k P_D=P_{kD}$. Note that the inequalities $u_i \geq 0$ imply 
\[
\chi^{\alpha} \in \mathbb C[x_1, \ldots, x_n] \quad \quad \makebox{for all}\quad \quad  
\alpha \in P_{kD} \cap \mathbb Z^n. 
\]
ii) Let $\chi^{\alpha}$ be a monomial, such that the exponent 
 $\alpha=(\alpha_1, \ldots, \alpha_n) \in P_{kD} \cap \mathbb Z^n$ is not contained in the polytope
\[
k(\Box_{l} \cap P_D)= \Box_{kl} \cap P_{kD}, 
\]
say $kl_1 < \alpha_1$. Then we define  $\beta_1:=\alpha_1 - kl_1$ and write $\chi^{\alpha}$ as a product 
\[
\chi^{\alpha}= \chi^{(\beta_1,\alpha_2,\ldots,\alpha_n)} \chi^{kl_1e_1}.
\]
\end{proof}

\begin{remark}
\makebox{$$}
\begin{itemize}
\item
Note that the inequalities $\langle u, w_{\rho} \rangle \geq - ku_{\rho}$ in the definition of the polyhedron 
\[
P_{kD}=\lbrace u \in \mathbb R^n ~ | ~ u_i \geq 0, \quad \langle u, w_{\rho} \rangle \geq - ku_{\rho}\rbrace
\]
are redundant if $u_{\rho} \geq 0$. 
\item 
For $D=\epsilon^{\ast}E$ we have 
$u_{\rho}=  \lambda_{\rho} \big(\langle v_{\rho}, e_1+ \ldots + e_n \rangle-1\big)$.
This integer is, according to Proposition \ref{discrepancies}, equal 
to the discrepancy of $E_{\rho}$ multiplied by $\lambda_{\rho}>0$. In particular, in the case of a canonical singularity, the ideal $\mathcal I_{\epsilon^{\ast}kE}$ is trivial, since
all $u_{\rho} \geq 0$. 
\item
The ideal $I_{kD}$ has a unique minimal basis, because it is a monomial ideal. 
\end{itemize}
\end{remark}

\begin{remark}
If we perform the star subdivision of the cone $\sigma$  along 
all rays generated by a primitive lattice point $v_{\rho}$ with 
\[
\langle v_{\rho}, e_1+ \ldots + e_n \rangle-1 < 0 
\]
we obtain a fan $\Sigma'$ 
that is not necessarily smooth. 
However, there is a subdivision of  $\Sigma'$ yielding a smooth fan $\Sigma$. Since the new rays 
$\rho \in \Sigma \setminus \Sigma'$ do not contribute to the polyhedra of 
$\epsilon^{\ast}kE$, there is no need to compute $\Sigma$ explicitly.  
\end{remark}

\noindent
From the description of the ideal $I_{kD}$, it follows that 
$(I_D)^k \subset I_{k D}$ for all positive integers $k$. 
However,  this inclusion is in general not an equality. The reason is that the
polytope $\Box_{l} \cap P_D$ may not contain enough lattice points. 
We can solve this problem by replacing $D$ with a high enough multiple:

\begin{proposition}\label{IsD^k=IskD}
Let $D$ be a divisor as in Proposition \ref{pushfor}. 
Then, there exists a positive integer $s$ such that 
\[
(I_{s D})^k=I_{sk D}  \quad \quad \makebox{for all} \quad \quad k \geq 1. 
\]
\end{proposition}

\begin{proof}
Let $l=(l_1,\ldots,l_n)$ be a tuple of positive integers such that $l_i \cdot e_i \in P_D$. According to  
Proposition \ref{pushfor} ii) the monomials $\chi^{\alpha}$ with
\[
\alpha \in k\big(\Box_{l} \cap P_D\big) \cap \mathbb Z^n
\]
generate $I_{kD}$ for all $k \geq 1$. Since the vertices of the polytope $\Box_{l} \cap P_D$ 
have rational coordinates, there is a positive integer $s'$ such that $s'\big(\Box_{l} \cap P_D\big)$ 
is a lattice polytope i.e. the convex hull of finitely many lattice points.
We define $s:=s'(n-1)$,  then $s\big(\Box_{l} \cap P_D\big)$ is a normal lattice polytope 
(see \cite{CLS}*{Theorem 2.2.12}), which means that 
\[
\big(ks'\big(\Box_{l} \cap P_D\big)\big) \cap \mathbb Z^n = k \big(s'\big(\Box_{l} \cap P_D\big) \cap \mathbb Z^n\big)
\quad \quad \makebox{for all } \quad \quad k \geq 1.
\]
Clearly, this implies $(I_{s D})^k=I_{sk D}$ for all $k \geq 1$. 
\end{proof}

\begin{remark}
According to the proof of Proposition \ref{IsD^k=IskD} we can take $s=(n-1)s'$ where $s'$ is the smallest positive integer such 
that all the vertices of $s'P_D$ have integral coordinates.
\end{remark}

{\scriptsize
%\begin{verbatim}
\begin{codice_magma}[caption={Computation of the ideal $I_{k \epsilon^{\ast} E}$ for the singularity $1/n(1,a,b)$\\ }]

// The first function determines the lattice points that we need to blow up according 
// to Computational Rem 5.6. It returns the primitive generators "w_rho" of these points
// according to Prop 5.3 and the discrepancy of the pullback divisor eps^{\ast} E.

Vectors:=function(n,a,b)
Ve:={};
for i in [1..n-1] do
	x:=i/n;
	y:=(i*a mod n)/n;
	z:=(i*b mod n)/n;
	d:=x+y+z-1;
		if d lt 0 then 
            lambda:=Lcm([Denominator(x),Denominator(y),Denominator(z)]);
            Include(~Ve,[lambda*x,lambda*y,lambda*z,lambda*d]);
		end if;
    end for;
    return Ve;
end function;

// The function "IntPointsPoly" determines a basis for the monomial ideal,  
// according to Proposition 5.4. However, this basis is not necessarily minimal. 
// The subfunction "MinMultPoint" is used to determine the cube in ii) of Proposition 5.4. 

MinMultPoint:=function(P,v)
    n:=1;
	while n*v notin P do 
	    n:= n+1;
	end while;
    return n;
end function;

IntPointsPoly:=function(n,a,b,k)
    L:=ToricLattice(3);
    La:=Dual(L);
    e1:=L![1,0,0]; e2:=L![0,1,0]; e3:=L![0,0,1]; 
    P:=HalfspaceToPolyhedron(e1,0) meet 
       HalfspaceToPolyhedron(e2,0) meet 
       HalfspaceToPolyhedron(e3,0);
    Vec:=Vectors(n,a,b);
    for T in Vec do 
	    w:=L![T[1],T[2],T[3]];
	    u:=T[4];
	    P:= P meet HalfspaceToPolyhedron(w,-k*u); 
    end for;
    multx:=MinMultPoint(P,La![1,0,0]); 	
    multy:=MinMultPoint(P,La![0,1,0]); 
    multz:=MinMultPoint(P,La![0,0,1]); 
    P:=P meet HalfspaceToPolyhedron(L![-1,0,0],-k*multx) meet 
		      HalfspaceToPolyhedron(L![0,-1,0],-k*multy) meet 
		      HalfspaceToPolyhedron(L![0,0,-1],-k*multz); 
    return Points(P);		 
end function;

// The next functions are used to find the (unique) minimal monomial basis of the ideal.

IsMinimal:=function(Gens)
    test:=true; a:=0;
	for a_i in Gens do
		for a_j in Gens do
	    	if a_i ne a_j then 
			    d:=a_j-a_i;
                if d.1 ge 0 and d.2 ge 0 and d.3 ge 0 then 
					a:=a_i; test:=false;
					break a_i;
				end if;
			end if;
		end for;
	end for;
    return test, a;
end function;

SmallerGen:=function(Gens,a)	
    Set:=Gens;	
    for b in Gens do
		d:=b-a;
		if d.1 ge 0 and d.2 ge 0 and d.3 ge 0 and b ne a then 
			Exclude(~Set,b);
		end if;
	end for;
    return Set;
end function;

MinBase:=function(n,a,b,k)
    F:=RationalField();
    PL<x1,x2,x3>:=PolynomialRing(F,3);
    test:=false;
    Gens:=IntPointsPoly(n,a,b,k);
    while test eq false do
	    test, a:=IsMinimal(Gens);
		if test eq false then 
			Gens:=SmallerGen(Gens,a);
		end if;
    end while;
    MB:={};
    for g in Gens do 
        Include(~MB,PL.1^g.1*PL.2^g.2*PL.3^g.3);
    end for;
    return MB;
end function;
\end{codice_magma}
}

\section{A Calabi-Yau 3-fold}\label{trueCY}

In this section we apply Theorem \ref{mainthm} to the first numerical Calabi-Yau threefold listed in Section \ref{list}, table \ref{figure CY}.

We start by giving an explicit description of the threefold by writing the canonical ring of the curve $C:=C_1 \cong C_2 \cong C_3$ and the group action on it.

We consider the hyperelliptic curve 
\[
C:=\lbrace y^2=x_0^{6} + x_1^{6} \rbrace \subset \mathbb P(1,1,3)
\]
of genus $2$, 
together with the  $\ZZ_6$-action generated by the automorphism $g$ defined by
\[
g((x_0:x_1:y))=(x_0:\omega x_1:y), \quad \quad \makebox{where} \quad \quad \omega:=e^{\frac{2\pi i}{6}}.
\]
By adjunction there is an isomorphism of graded rings between $R(C,K_C):=\bigoplus_{d} H^0(C,\sO_C(dK_C)) $ and $\CC[x_0,x_1,y]/(y^2-x_0^6-x_1^6))$, where $\deg x_i=1$ and $\deg y=3$. 

\begin{lemma}\label{actiononR}
The action of $g$ on $R(C,K_C)$ induced by the pull-back of holomorphic differential forms is
\begin{align*}
x_0 \mapsto &\omega x_0&
x_1 \mapsto &\omega^2 x_1&
y \mapsto &\omega^3 y=-y\\
\end{align*}
\end{lemma}
\begin{proof}
Consider the smooth affine chart $x_0 \neq 0$ with local coordinates $u:=\frac{x_1}{x_0}$ and  $v:=\frac{y}{x_0^3}$. In this chart $C$ is the vanishing locus of $f:=v^2-u^6-1$. By adjunction the monomials $x_0,x_1,y \in R(C,K_C)$ 
correspond respectively to the forms that, in this chart, are
\begin{align*}
x_0 \mapsto&\frac{du}{\frac{\partial f}{\partial v}}= \frac{du}{2v} &
x_1 \mapsto &u \frac{du}{2v}& 
y \mapsto & v \left (\frac{du}{2v}\right)^{\otimes 3}& 
\end{align*}
The statement follows since $g$ acts on the local coordinates as $(u,v) \mapsto (\omega u, v)$.
\end{proof}
\begin{proposition}\label{firstex}
The threefold $X:=C^3/\ZZ_6$, where the group $\ZZ_6$ acts as above on each copy of $C$, is a numerical Calabi-Yau threefold. 

There are $8$ non canonical singularities on $X$, all of type
$ \frac{1}{6}(1,1,1)$.
\end{proposition}

\begin{proof}
The points on $C$ with non-trivial stabilizer subgroup of $\ZZ_6$ are the four points $p_0,p_1,p_2,p_3$ with the following weighted homogeneous coordinates $(x_0:x_1:y)$:
\begin{align*}
p_0=&(1:0:1) &
p_1=&(1:0:- 1) &
p_2=&(0:1:1) &
p_3=&(0:1:- 1) \\
\end{align*}
In the table below, for each point $p_j$, we give a generator of its stabilizer, and the action of the generator on a local parameter of the curve $C$ near $p_j$.

{\small
\begin{center}%\label{chartable} 
\setlength{\tabcolsep}{1pt}
\begin{tabular}{|c |c | c | c |}
\hline
point & $ ~ p_{0/1}=(1:0:\pm 1) ~ $ & $ ~ p_{2/3}=(0:1:\pm 1) ~ $   \\
\hline
generator of the stabilizer  & $ g  $ & $ g^2 $   \\
\hline
local action  & $x \mapsto \omega x$ & $x \mapsto \omega^4 x$  \\
\hline
\end{tabular}
\end{center}
}

$p_0$ and $p_1$ are then stabilized by the whole group $\ZZ_6$, forming then two orbits of cardinality $1$, whereas  $p_2$ and $p_3$ are stabilized by the index two subgroup of $\ZZ_6$, and form a single orbit. 

Consequently the points with nontrivial stabilizer are the $64$ points $p_{i_1}\times p_{i_2}\times p_{i_3}$ forming 
$8$ orbits of cardinality $1$, the points $p_{i_1}\times p_{i_2}\times p_{i_3}$ with $i_j \in \{0,1\}$, and
 and $28$ of cardinality $2$. So $C^3/\ZZ_6$ has $36$ singular points:
 \begin{itemize}
 \item $8$ singular points of type  $ \frac{1}{6}(1,1,1)$, the classes of the points $p_{i_1}\times p_{i_2}\times p_{i_3}$ with $i_j \in \{0,1\}$: these are not canonical;
 \item $4$ singular points of type  $ \frac{1}{3}(1,1,1)$, the classes of the points $p_{i_1}\times p_{i_2}\times p_{i_3}$ with $i_j \in \{2,3\}$: these have a crepant resolution;
 \item $24$ singular points of type  $ \frac{1}{3}(1,1,2)$, the classes of the remaining points $p_{i_1}\times p_{i_2}\times p_{i_3}$: these are terminal singularities.
 \end{itemize}

We prove now that a resolution $\rho \colon \widehat{X} \to X=C^3/\ZZ_6$ has  invariants $p_g(\widehat{X})=1$,  $q_1(\widehat{X})=q_2(\widehat{X})=0$ using representation theory and the fact that 
\[
H^0(\widehat{X}, \Omega_{\widehat{X}}^i) \simeq H^0(C^3,\Omega_{C^3}^i)^G. 
\]

By Lemma \ref{actiononR}
the character of the natural representation
$\varphi \colon \ZZ_6 \to GL\big(H^0(K_C)\big)$ is 
$\chi_{\varphi} = \chi_{\omega} + \chi_{\omega^2}$.
By K\"unneth's formula  the characters $\chi_i$ of the $\ZZ_6$ 
representations on  $H^0(C^3,\Omega_{C^3}^i)$ are respectively
\begin{align*}
\chi_3 =& \chi_{\varphi}^3 &
\chi_2 = &3 \chi_{\varphi}^2 &
\chi_1 =& 3 \chi_{\varphi}. 
\end{align*}

The claim follows, since $\chi_3$ contains exactly one copy of the trivial character whereas 
$\chi_2$ and $\chi_1$ do not contain the trivial character at all.  
\end{proof}
We write coordinates
\[
\left(
(x_{01}:x_{11}:y_1),
(x_{02}:x_{12}:y_2),
(x_{03}:x_{13}:y_3)
\right)
\]
on $\PP(1,1,3)^3$, so that $C^3$ is the locus defined by the ideal $\left(y^2_j-x_{0j}^6-x_{1j}^6, j=1,2,3 \right)$.

K\"unneth's formula yields a basis for $H^0(dK_{C^3})$: 
\[
\bigg\lbrace \prod_{i=1}^3 x_{0 i}^{a_i}  x_{1 i}^{b_i} y_i^{c_i} ~ \big\vert ~ 
a_i+b_i+3c_i =d,  \quad c_i =0,1  \bigg\rbrace. 
\]
on which $g$ acts as
\[
\prod_{i=1}^3 x_{0 i}^{a_i}  x_{1 i}^{b_i} y_i^{c_i}  \mapsto \omega^{\sum_i (a_i+2bi+3c_i)} \prod_{i=1}^3 x_{0 i}^{a_i}  x_{1 i}^{b_i} y_i^{c_i} 
\]

By the proof of Proposition  \ref{firstex}, writing $p_i=(1:0:(-1)^i)\in \mathbb P(1,1,3)$ for $i =0,1$, the eight points 
\[
p_{i_1}\times p_{i_2}\times p_{i_3} , \qquad  i_j =0,1
\]
are precisely those that descend to the eight singularities of type $\frac{1}{6}(1,1,1)$. 

To determine the plurigenenera of $X$ we need the following lemma.
\begin{lemma}\label{Id=P3d}
For all $d \geq 1$, the sheaf of ideals $\sI_d$ equals $\sP^{3d}$, where $\sP$ is the ideal of the reduced scheme 
$\{p_{i_1}\times p_{i_2}\times p_{i_3} |  i_j =0\}$.
\end{lemma}
\begin{proof}
As already mentioned, all non-canonical singularities are of type $\frac16(1,1,1)$.
These singularities are resolved by a single toric blowup along the ray $\rho$ generated by 
$v:=\frac16(1,1,1)$. The polyhedron associated to the divisor $\epsilon^{\ast}dE=-3dF_{\rho}$ is 
\[
P_{-3dF_{\rho}}=\lbrace u\in \mathbb R^3 ~ | ~ u_i \geq 0, \quad  u_1+u_2+u_3 \geq 3d \rbrace,
\]
so the corresponding ideal is just the $3d$-th power of the maximal ideal. 
\end{proof}
Then we can prove
\begin{proposition}
$X=C^3/\ZZ_6$ is birational to a Calabi-Yau threefold.
\end{proposition}

\begin{proof}
By Proposition \ref{plurigenera==>CY} we only need to prove that all plurigenera are equal to $1$, so, by Theorem \ref{mainthm}, that, $\forall d \geq 1$, 
\[
H^0\left(C^3,\sO_{C^3}(dK_{C^3}) \otimes \mathcal{I}_d \right)^G \cong \CC
\] 

The vector space $H^0\left(C^3,\sO_{C^3}(dK_{C^3}) \right)$ is contained in the $\ZZ^3$-graded ring
\[
R:=\CC[x_{01},x_{11},y_1,x_{02},x_{12},y_2,x_{03},x_{13},y_3]/\left(y_i^2-x_{i0}^6-x_{i1}^6, i=1,2,3\right)
\]
with gradings
\begin{align*}
\deg x_{01}=&(1,0,0)&
\deg x_{11}=&(1,0,0)&
\deg y_{1}=&(2,0,0)\\
\deg x_{02}=&(0,1,0)&
\deg x_{12}=&(0,1,0)&
\deg y_{2}=&(0,2,0)\\
\deg x_{03}=&(0,0,1)&
\deg x_{13}=&(0,0,1)&
\deg y_{3}=&(0,0,2)\\
\end{align*}
as the subspace $R_{d,d,d}$ of the homogeneous elements of multidegree $(d,d,d)$. By Lemma \ref{actiononR} the natural action of $G$ on $H^0\left(C^3,\sO_{C^3}(dK_{C^3}) \right)$ is induced by the restriction of the following action of its generator $g$ on $R$:
\begin{align*}
x_{0i} \mapsto& \omega x_{0i}&
x_{1i} \mapsto& \omega^2 x_{1i}&
y_{i} \mapsto& \omega^3 y_{i}
\end{align*}

By Lemma \ref{Id=P3d}, since the elements of $R$ vanishing on the reduced scheme 
$\{p_{i_1}\times p_{i_2}\times p_{i_3} |  i_j =0\}$ form the ideal $ (x_{11},x_{21},x_{31})$
 \[
H^0\left(C^3,\sO_{C^3}(dK_{C^3}) \otimes \mathcal{I}_d \right) = R_{d,d,d} \cap (x_{11},x_{12},x_{13})^{3d}=
\langle (x_{11}x_{12}x_{13})^d\rangle
\]
is one dimensional. 

Since its generator $x_{11}x_{12}x_{13}$ is $G-$invariant, the proof is complete.
\end{proof}

\section{A fake Calabi-Yau 3-fold}\label{fakeCY}
We consider the hyperelliptic curves 
\[
C_2:=\lbrace y^2 =x_0x_1(x_0^4+x_1^4) \rbrace\subset \mathbb P(1,1,3) \qquad \makebox{and} \qquad 
C_3:=\lbrace y^2 =x_0^8 + x_1^8 \rbrace\subset \mathbb P(1,1,4)
\]
of respective genus two and three, together with the $\ZZ_8$-actions 
$g(x_0:x_1:y)=(x_0:\omega^2x_1:\omega y)$ 
on $C_2$ and 
$g(x_0:x_1:y)=(x_0:\omega x_1: y)$ 
on $C_3$, where $\omega=e^{\frac{2\pi i}{8}}$.

\begin{proposition}
The threefold $X=(C_2^2 \times C_3)/G$, 
where $G=\ZZ_8$ acts diagonally, is a numerical Calabi-Yau threefold. 
$X$ has exactly $44$ singular points and more precisely
\[
6 \times \frac{1}{8}(1,1,3), \quad 
2 \times \frac{1}{8}(1,1,1), \quad 
3 \times \frac{1}{4}(1,1,3), \quad 
1 \times \frac{1}{4}(1,1,1).
\quad 
32  \times \frac{1}{2}(1,1,1). 
\]
\end{proposition}

\begin{proof}
The points with non-trivial stabilizer on $C_2$  are 
$q_0:=(0:1:0)$ and  $q_1:=(1:0:0)$ with the full group as stabilizer and 
the points 
\[
p_i:=(1:x_i:0), \quad \quad \makebox{where} \quad \quad  x_i^4 =-1
\]
with stabilizer $\langle g^4 \rangle \cong \ZZ_2$. 

Next, we compute the local action around the points $p_i$ and $q_i$.

The points $q_1$ and $p_i$ are contained in the smooth affine chart $x_0 \neq 0$ of $\PP(1,1,3)$, with affine coordinates $u=\frac{x_1}{x_0}$ and $v=\frac{y}{x^2_0}$. Here, the curve is the vanishing locus of the polynomial $f:=v^2-u^5-u$ and the group acts via $
(u,v) \mapsto (\omega^2u,\omega v)$.

Since $\frac{\partial f}{\partial u}(q_1)=-1$ and $\frac{\partial f}{\partial v}(p_i)=4$, by the implicit function theorem, $v$ is a local parameter for $C_2$ near these points. In particular $g$ acts around
$q_1$ as the multiplication by $\omega$ and $g^4$ acts around $p_i$ as the multiplication by  $\omega^4=-1$.

A similar computation on  the affine chart $x_1 \neq 0$ shows
that $g$ acts around $q_0$ as the multiplication by $\omega^3$. 
The table below summarizes our computation. 

\bigskip

{\small
\begin{center}%\label{chartable} 
\setlength{\tabcolsep}{1pt}
\begin{tabular}{|c |c | c | c |}
\hline
point & $ ~ q_0=(0:1:0) ~ $ & $ ~ q_1=(1:0:0) ~ $ & $ ~ p_i=(1:x_i:0) ~ $  \\
 &   &  & $x_i^4=-1$  \\
\hline
Stab  & $\langle g \rangle $ & $\langle g \rangle $ & $\langle g^4 \rangle$  \\
\hline
$~ \makebox{local action} ~ $ & $x \mapsto \omega^3 x $ & $x \mapsto \omega x$ & $x \mapsto -x$ \\
\hline
\end{tabular}
\end{center}
}

\bigskip
\noindent 
Similarly, for $C_3$, we obtain 

\bigskip 

{\small
\begin{center}%\label{chartable} 
\setlength{\tabcolsep}{1pt}
\begin{tabular}{|c |c | c |}
\hline
points & $ ~ s_1=(1:0:1), ~ $  & $ ~ s_3=(0:1:1), ~ $  \\
 & $ ~ s_2=(1:0:-1) ~ $  & $ ~ s_4=(0:1:-1) ~ $  \\
\hline
Stab  & $\langle g \rangle $ &  $\langle g^2 \rangle$   \\
\hline
$~ \makebox{local action} ~ $  & $x \mapsto \omega x$ & $x \mapsto \omega^6 x$  \\
\hline
\end{tabular}
\end{center}
}

Then the 
diagonal action on $C_2^2 \times C_3$ admits $6\cdot 4 \cdot 4=144$ points with non-trivial stabilizer. 
The $8$ points of the form 
\[
q_i \times q_j \times s_k, \quad  \makebox{where} \quad i,j \in \lbrace 0,1 \rbrace \quad  \makebox{and} \quad 
k \in  \lbrace 1,2 \rbrace. 
\]
are stabilized by the full group. Therefore, they are mapped to $8$ singular points on the quotient.
These singularities are 
\[
\begin{cases}
   2 \times \frac{1}{8}(1,1,1) & \text{for } i=j=0 \\
  4 \times \frac{1}{8}(1,1,3), & \text{for }i\neq j \\
   2 \times \frac{1}{8}(1,3,3) & \text{for } i=j=1
  \end{cases}
\]
The $8$ points 
\[
q_i \times q_j \times s_k, \quad  \makebox{where} \quad i,j \in \lbrace 0,1 \rbrace \quad  \makebox{and} \quad 
k \in  \lbrace 3,4 \rbrace
\]
have $\langle g^2 \rangle \cong \ZZ_4$ as stabilizer. These map to 4 singular points on the quotient:

\[
\begin{cases}
  1 \times \frac{1}{4}(1,1,1) & \text{for } i=j=0 \\
  3 \times \frac{1}{4}(1,1,3), & \text{else}
  \end{cases}
\]
The remaining $128$ points have stabilizer ${\mathbb Z}_2$. 
These points yield $32$ terminal singularities of type $\frac{1}{2}(1,1,1)$ on the quotient.

To show that $X$ is numerical Calabi-Yau, we verify that 
\[
p_g(\widehat{X}) =1,\qquad  \makebox{and} \qquad q_2(\widehat{X})= q_1(\widehat{X}) =0 
\]
for a resolution $\widehat{X}$ of $X$ like in the proof of Proposition \ref{firstex}. 
\end{proof}
This example is not birational to a Calabi-Yau threefold. 

\begin{proposition}\label{not CY!}
Let $\rho \colon \widehat{X} \to X$ be a resolution of the singularities of $X$ and $Z$ be a minimal model of $\widehat{X}$. 

Then $Z$ is not Calabi-Yau. 
\end{proposition}

\begin{proof}
We show that $P_2(\widehat{X}) \geq 3$. A monomial basis of $H^0(2K_{C_2^2 \times C_3})$ is
\[
\bigg\lbrace \prod_{i=1}^3 x_{0 i}^{a_i}   x_{1 i}^{b_i} y_i^{c_i} ~ \big\vert ~ 
a_1+b_1+3c_1= a_2+b_2+3c_2 =2, ~ a_3 + b_3 + 4c_3 = 4 \bigg\rbrace. 
\]
The table below displays all points on $C_2^2 \times C_3$ with non-trivial stabilizer, that descend to a non-canonical singularity and the germ of the plurisection 
\[
\prod_{i=1}^3 x_{0 i}^{a_i} \cdot  x_{1 i}^{b_i} \cdot  y_i^{c_i}
\]
in local coordinates up to a unit as well as the stalks of the ideal $\sI_2$ in these points. 

\bigskip

{\small
\begin{center}%\label{chartable} 
\setlength{\tabcolsep}{1pt}
\begin{tabular}{|c |c | c | c |}
\hline
point & ~ singularity ~ & ~ germ ~ & ~ stalk  ~ \\
\hline
 &   &  &  \\
~ $(q_0,q_0,s_{1/2})$ ~ & ~ $\frac{1}{8}(1,1,3)$ ~ & ~ $y_1^{2a_1+c_1}y_2^{2a_2+c_2}x_3^{b_3}$ ~ &  
~ $\big((y_1,y_2)^3 +  (x_3) \big)^2$ ~ \\
 &   &  &  \\
\hline
 &   &  &  \\
~ $(q_1,q_1,s_{1/2})$ ~ & ~ $\frac{1}{8}(1,1,1)$ ~ & ~ $y_1^{2b_1+c_1}y_2^{2b_2+c_2}x_3^{b_3}$ ~ &  
~ $(y_1,y_2,x_3)^{10}$ \\
 &   &  &  \\
\hline
 &   &  &  \\
~ $(q_0,q_1,s_{1/2})$ ~ & ~ $\frac{1}{8}(1,3,3)$ ~ & ~ $y_1^{2a_1+c_1}y_2^{2b_2+c_2}x_3^{b_3}$ ~ & 
~  $\big((y_2,x_3)^3 +  (y_1) \big)^2$ \\
 &   &  &  \\
\hline
 &   &  &  \\
~ $(q_1,q_0,s_{1/2})$ ~ & ~ $\frac{1}{8}(1,3,1)$ ~ & ~ $y_1^{2b_1+c_1}y_2^{2a_2+c_2}x_3^{b_3}$ ~ & 
~ $\big((y_1,x_3)^3 +  (y_2) \big)^2$  \\
 &   &  &  \\
\hline
&   &  &  \\
~ $(q_0,q_0,s_{3/4})$ ~ & ~ $\frac{1}{4}(1,1,1)$ ~ & ~ $y_1^{2a_1+c_1}y_2^{2a_2+c_2}x_3^{a_3}$ ~ & 
~ $(y_1,y_2,x_3)^2$ \\
 &   &  &  \\
\hline
\end{tabular}
\end{center}
}

With the help of  MAGMA we found the following three monomial sections of $H^0(2K_{C_2^2 \times C_3} \otimes \sI_2)$:
\[
x_{11}^2x_{12}^2x_{03}^2x_{13}^2, \quad x_{01}x_{11}x_{12}^2x_{13}^4 \quad \makebox{and} \quad 
x_{11}^2x_{02}x_{12}x_{13}^4.
\] 
Using the same argument as in Lemma \ref{actiononR}, 
we obtain the action on the canonical ring $R(C_2,K_{C_2})$ as
\[
x_0 \mapsto \eta^2 x_0, \quad x_1 \mapsto \eta^6 x_1, \quad y \mapsto \eta^8 y, \quad \makebox{where} \quad \eta^2=\omega
\]
and on $R(C_3,K_{C_3})$ as
\[
x_0 \mapsto \eta x_0, \quad x_1 \mapsto \eta^3 x_1, \quad y \mapsto \eta^4 y.
\]
We conclude that the three sections above are also $\ZZ_8$ invariant, in particular $P_2(\widehat{X}) \geq 3$.
\end{proof}

Note that each of the three monomials in the proof of Proposition \ref{not CY!} contains a variable that does not appear in the other two. This implies that the subring of the canonical ring of $\widehat{X}$ generated by the three monomials is isomorphic to the ring of polynomials in three variables. In particular $\kod(\widehat{X})\geq 2$.

\section{Some minimal surfaces of general type}\label{MinSurfs}
In this section we construct some product-quotient surfaces with several singular points and investigate their minimality.

 The construction is as follows.

\begin{definition}
\noindent Let $a,b\in\mathbb{N}$ such that $\gcd(ab,1-b^2)=1$, $ab\geq 4$ and $b\geq 3$.
Define $n=ab$ and let $1\leq e\leq n-1$ be the unique integer such that $e\cdot(1-b^2)\equiv_n 1$ (i.e. $e$ represents the inverse modulo $n$ of $1-b^2$). For example, one can take $a=b\geq 3$. Define 
$$\omega=\e^{\frac{2\pi i}{n}}\qquad \mbox{ and }\qquad \lambda=\e^{\frac{2\pi i}{n(n-3)}}.$$
Consider the Fermat curve $C$ of degree $n$ in $\PP^2$, i.e. the plane curve \[x_0^n+x_1^n+x_2^n=0.\]
where $x_i$ are projective coordinates on $\PP^2$. Consider the natural action $\rho_1$ of  $G:=\Z_n\oplus \Z_n$ on $C$ generated by 
\begin{align}\label{ro1}
g_1\cdot(x_0:x_1:x_2)=&(\lambda x_0:\lambda \omega x_1 :\lambda x_2)& h_1\cdot(x_0:x_1:x_2)=&(\lambda x_0:\lambda x_1 :\lambda \omega x_2).
\end{align}
Define
\begin{equation}
g_2:=g_1h_1^{b}\qquad h_2:=g_1^{-b}h_1^{-1}\qquad (\mbox{ and } k_2=g_2^{-1}h_2^{-2}=g_1^{b-1}h_1^{1-b}).
\end{equation}
Under the above assumptions, $g_2$ and $h_2$ are generators of $G$, inducing  a second $G$-action $\rho_2$ on $C$ by 
\begin{align*}
g_2\cdot (x_0:x_1:x_2):=&g_1\cdot (x_0:x_1:x_2)& h_2\cdot (x_0:x_1:x_2):=&h_1\cdot (x_0:x_1:x_2)
\end{align*}
The diagonal action $\rho_1 \times \rho_2$ on $C \times C$ gives a product quotient surface $\widehat{X}_{a,b}$ with quotient model $X_{a,b}$. 
\end{definition}

 The action $\rho_1$ has $3$ orbits where the action is not free:
\begin{equation*}
\begin{array}{c}
\Fix(g_1)=\{(1:0:-\eta) \quad|\quad \eta^n=1\}\\
\Fix(h_1)=\{(1:-\eta:0) \quad|\quad \eta^n=1\}\\
\Fix(k_1)=\{(0:1:-\eta) \quad|\quad \eta^n=1\}
\end{array}
\end{equation*}
respectively stabilized by $\langle g_1\rangle, \langle h_1\rangle$ and $\langle k_1:=g_1^{-1}h_1^{-1} \rangle$.
Notice $g_1=g_2^{e}h_2^{2b}$ and $h_1=g_2^{-eb}h_2^{-e}$.
The following relations hold:
\begin{align*}\langle g_1\rangle\cap \langle g_2\rangle=& \langle g_1^a\rangle& \langle g_1\rangle\cap \langle h_2\rangle=& \langle 1\rangle& \langle g_1\rangle\cap \langle k_2\rangle=& \langle 1\rangle\\
\langle h_1\rangle\cap \langle h_2\rangle=& \langle h_1^a\rangle& \langle h_1\rangle\cap \langle k_2\rangle=& \langle 1\rangle& \langle k_1\rangle\cap \langle k_2\rangle=& \begin{cases}
\langle 1\rangle & \mbox{ if } n \mbox{ is odd}\\
\langle (g_1h_1)^{n/2} \rangle & \mbox{ if } n \mbox{ is even}
\end{cases}.
\end{align*}
The only points of $C \times C$ with non trivial stabilizer are
\begin{equation*}
\begin{array}{c|c|c|c|c}
\mbox{Fixed points} & \mbox{\#Points} & \mbox{Stabilizer} & \mbox{Type of singularity on } X & \\ \hline
\Fix(g_1)^2 & n^2 & \langle g_1^a\times g_1^a\rangle\simeq \Z_b & \frac{1}{b}(1,1) & \mbox{Any } n\\
\Fix(h_1)^2 & n^2 & \langle h_1^a\times h_1^{-a}\rangle\simeq \Z_b & \frac{1}{b}(1,b-1) & \mbox{Any } n\\
\Fix(k_1)^2 & n^2 & \langle k_1^{n/2}\times k_1^{n/2}\rangle\simeq \Z_2 & \frac{1}{2}(1,1) & n \mbox{ even}\\
\end{array}
\end{equation*}
In particular, the only non canonical singularities of $X_{a,b}$ are $b$ points of type $\frac{1}{b}(1,1)$. \\

Since $C/G\simeq \PP^1$ then  $q(\widehat{X}_{a,b})=0$. Moreover, we have, by the formulas in \cite{BaPi12},
\begin{equation*}
 K_{X_{a,b}}^2=\frac{8(g(C)-1)^2}{\#G}=2(n-3)^2
\end{equation*}
and, as we have exactly $b$ singular points of type $\frac1b (1,1)$ ,
\[
r^*K_{X_{a,b}}=K_{\widehat{X}_{a,b}}+\frac{b-2}{b}(E_1+\cdots+E_b)
\]
where $E_i$ are the exceptional divisors introduced by the resolution over the non-canonical points. These are disjoint rational curves with selfintersection $-b$ so
\begin{equation}
K_{\widehat{X}_{a,b}}^2=K_{X_{a,b}}^2-b^2\frac{(b-2)^2}{b^2}=2(n-3)^2-(b-2)^2.
\end{equation}

\begin{remark}
Notice that $K_{\widehat{X}_{a,b}}^2\geq 2$ for all $(a,b)$ satisfying our assumptions, unless $(a,b)\in\{(1,4),(1,5)\}$.
\end{remark}

There is an isomorphism between $H^0(K_C)$ and $$H^0(\OO_{C}(n-3))=H^0(\OO_{\PP^2}(n-3))=\CC[x_0,x_1,x_2]_{n-3}.$$ 
Then $\rho_1$ induces a $G$-action on $H^0(\omega_C)$ via pull-back of holomorphic forms on $C$. We wrote $\rho_1$ so that this action coincides with the natural action induced by (\ref{ro1}) on monomials of degree $n-3$. Explicitly, if $m_0+m_1+m_2=n-3$ we have
\begin{multline*}
g_1\cdot x_0^{m_0}x_1^{m_1}x_2^{m_2}=(g_1^{-1})^*(x_0^{m_0}x_1^{m_1}x_2^{m_2})=\\
=\lambda^{-m_0-m_1-m_2}\omega^{-m_1} x_0^{m_0}x_1^{m_1}x_2^{m_2}=\omega^{-m_1-1} x_0^{m_0}x_1^{m_1}x_2^{m_2}
\end{multline*}
and
\begin{multline*}
h_1\cdot x_0^{m_0}x_1^{m_1}x_2^{m_2}=(h_1^{-1})^*(x_0^{m_0}x_1^{m_1}x_2^{m_2})=\\
=\lambda^{-m_0-m_1-m_2}\omega^{-m_2} x_0^{m_0}x_1^{m_1}x_2^{m_2}=\omega^{-m_2-1} x_0^{m_0}x_1^{m_1}x_2^{m_2}.
\end{multline*}

\noindent The canonical action induced by $\rho_2$ and the bicanonical action follow accordingly. Working as in the previous sections we computed $p_g(\widehat{X}_{a,b})=h^0(K_{C \times C})^G$ and $h^0(2K_{C \times C})^G$ for the case $a=b$. The values are all in Table \ref{TABLESURF}. We stress that for $a=b\geq 3$ we always get $K^2_{\widehat{X}_{a,b}} >0$ and $p_g(\widehat{X}_{a,b})>0$ so $\widehat{X}_{a,b}$ is of general type. 

\noindent As the only non-canonical singular points are of type $\frac{1}{b}(1,1)$ we have $$P_2(\widehat{X}_{b,b})=H^0(2K_{\widehat{X}_{b,b}})\simeq H^0(K_{C \times C}\otimes \mathcal{I}_{R_{nc}}^{2b-4})^G$$
where $\mathcal{I}_{R_{nc}}$ is the ideal sheaf of functions vanishing at order at least $2b-4$ in all the points of 
$$R_{nc}=\Fix(g_1)^2=\{(1:0:-\eta_1)\times (1:0:-\eta_2) \quad |\quad \eta_1^n=\eta_2^n=1\}.$$

\noindent Using the embedding of $C \times C$ in $\mathbb{P}^2 \times \mathbb{P}^2$ we have $H^0(2K_{C \times C})=H^0(\OO_{C \times C}(2n-6,2n-6))$. To simplify the computation, we just look for the invariant monomials with the right vanishing order on $R_{nc}$: in principle their number is only a lower bound for $P_2(\widehat{X}_{b,b})$; the vanishing order of 
$x_0^{m_0}x_1^{m_1}x_2^{m_2}y_0^{n_0}y_1^{n_1}y_2^{n_2}$
with $0\leq m_1,n_1,m_1+m_2,n_1+n_2\leq 2n-6$ and $0\leq m_2,n_2\leq n-1$ equals  $m_1+n_1$.\\

We prove

\begin{prop}
Assume $a=b\geq 3$. Then $H^0(2K_Y\otimes \mathcal{I}_{R_{nc}})^G$ is generated by invariant monomials.
Moreover, the codimension of $H^0(2K_Y\otimes \mathcal{I}_{R_{nc}})^G$ in $H^0(2K_Y)^G$ is $b(b-3)$.
\end{prop}

\begin{proof}
We give only a sketch of the proof.

The invariant bicanonical monomials are those of the form $x_0^{m_0}x_1^{m_1}x_2^{m_2}y_0^{n_0}y_1^{n_1}y_2^{n_2}$ with
\begin{equation}\label{TheSystem}
\begin{cases}
(I) :\hphantom{(II)} m_1+4+2b+n_1+bn_2\equiv_n 0\\
(II) :\hphantom{(I)} m_2-2b-bn_1-n_2\equiv_n 0\\
0\leq m_1,n_1,m_1+m_2,n_1+n_2\leq 2n-6\\
0\leq m_2,n_2\leq n-1\\
m_0+m_1+m_2=n_0+n_1+n_2=2b-6
\end{cases}
\end{equation}

We prove that if $\nu:=m_1+n_1 \leq 2b-4$ then $b\geq 4$ and $\nu=b-4$. Hence the Proposition is true for $b=3$ and we can assume $b \geq 4$. In this case we solve (\ref{TheSystem}) under the assumption $\nu=b-4$, finding  $b-3$ possibilities for the pair $(m_1,n_1)$. We denote by $W_b$ the vector space generated by the invariant bicanonical monomials and with $W_b^{(m_1,n_1)}$ its subspace generated by monomials which have assigned exponents for the variables $x_1$ and $y_1$. Monomials in $W_b^{(m_1,n_1)}$ satisfy $m_2,n_2\in \{-3+kb\,|\, 1\leq k\leq b\}$ and if $m_2=-3+k_mb$ and $n_2=-3+k_nb$ then $k_m\equiv_b k_n+n_1+2$. Using these observations, we  prove that the dimension of $W_b^{(m_1,n_1)}$ is $b$ which implies then that $W_b$ has codimension $b(b-3)$ in $H^0(2K_Y)^G$. \\

\noindent It remains to prove that a polynomial whose monomials have $\nu=b-4$ cannot vanish in all the points of $R_{nc}$ with order at least $2b-4$. A polynomial $p$ in $W_b$ is a linear combination of polynomials living in $W_b^{(m_1,n_1)}$. In affine coordinates $z_i=x_i/x_0$, $w_i=y_i/y_0$ we have 
$$p=\sum_{j=0}^{b-4} z_1^{j}w_1^{b-4-j}p_j(z_2,w_2)$$
with $z_1^{j}w_1^{b-4-j}p_j(z_2,w_2)\in W_b^{(m_1,n_1)}$.
In the second part of the proof we prove that if a polynomial $p\in W_b$ vanishes at order at least $2b-4$ in the points of $R_{nc}$ then, necessarily $p_j(-\eta,-\mu)=0$ for all pairs of $n$-th roots of $1$. This is obtained by implicit differentiation and by keeping track of the order of vanishing of the various terms of the sum. \\

\noindent In the third and final part of the proof we prove that if 
$z_1^{j}w_1^{b-4-j}q(z_2,w_2)\in W_b^{(m_1,n_1)}$ and $q(-\eta,-\mu)=0$ for enough pairs of $n$-th roots of $1$ then $q$ is actually $0$. More precisely, we prove that if $q(-1,\mu_i)=0$ for $\mu_1,\dots,\mu_b$ with $\mu^n=1$ and $\mu_i^b\neq \mu_j^b$ for $i\neq j$, then $q=0$. This can be seen as follows. If $[x]_b$ is the only representative of $x$ modulo $b$ in the range $[1,b]$ one can choose the monomials $f_k=z_1^{m_1}w_1^{b-4-m_1}q_k$ with
$$q_k=(-1)^{b(k+[k+2+n_1]_b)} z_2^{-3+b[k+2+n_1]_b} w_2^{-3+kb}\qquad 1\leq k\leq b$$ as basis for $W_{b}^{(m_1,n_1)}$. The coefficient is simply to get easier computations.
If $q=\sum_{k=1}^b \lambda_k q_k$ then
$$q(-1-\mu)=\mu^{-3}\sum_k \lambda_k(\mu^{b})^k.$$
We know that $q(-1,-\mu_i)=0$ for $\mu_1,\dots,\mu_b$. Then either $\lambda_k=0$ for all $k$ or the matrix $A=((\mu_i^b)^k)_{1\leq i,k\leq b}$ has determinant $0$. But $A$ is a Vandermonde-type matrix associated to $\{\mu_1^b,\dots,\mu_b^b\}$ and its determinant is zero if and only if there is a pair $(i,j)$ with $i\neq j$ such that $\mu_i^b=\mu_j^b$. But this contradicts the hypothesis so we have, finally, $p=0$.
\end{proof}

\noindent Having a way to compute $P_2$ also means that we have a way to determine whether our surfaces are minimal or not. Indeed, by Proposition \ref{vol}, we have $$ \vol (K_{\widehat{X}_{a,b}})=P_2(\widehat{X}_{a,b})-\chi(\OO_S)\geq K_{\widehat{X}_{a,b}}^2$$ with equality if and only if $S$ is already minimal. Here we summarize the invariants for the product-quotient surfaces 
obtained for $3 \leq a=b \leq 12$. 
\begin{equation*}
\label{TABLESURF}
\begin{array}{c||c|c|c|c|c|c|c|c|c}
\scriptstyle b &\scriptstyle g(C) &\scriptstyle K_{X_{b,b}}^2 &\scriptstyle K_{\widehat{X}_{b,b}}^2 &\scriptstyle p_g(\widehat{X}_{b,b}) & \scriptstyle\chi(\OO_{\widehat{X}_{b,b}}) &\scriptstyle h^0(2K_{C \times C})^G &\scriptstyle P_2(\widehat{X}_{b,b}) &\scriptstyle \vol K_{\widehat{X}_{b,b}} & \scriptstyle \vol K_{\widehat{X}_{b,b}} -K_{\widehat{X}_{b,b}}^2\\ \hline \hline
3 & 28 & 72 & 71 & 9 & 10 & 81 & 81 & 71 & 0\\
4 & 105 & 338 & 334 & 43 & 44 & 382 & 378 & 334 & 0\\
5 & 276 & 968 & 959 & 122 & 123 & 1092 & 1082 & 959 & 0\\
6 & 595 & 2178 & 2162 & 274 & 275 & 2455 & 2437 & 2162 & 0\\
7 & 1128 & 4232 & 4207 & 531 & 532 & 4767 & 4739 & 4207 & 0\\
8 & 1953 & 7442 & 7406 & 933 & 934 & 8380 & 8340 & 7406 & 0\\
9 & 3160 & 12168 & 12119 & 1524 & 1525 & 13698 & 13644 & 12119 & 0\\
10 & 4851 & 18818 & 18754 & 2356 & 2357 & 21181 & 21111 & 18754 & 0\\
11 & 7140 & 27848 & 27767 & 3485 & 3486 & 31341 & 31253 & 27767 & 0\\
12 & 10153 & 39762 & 39662 & 4975 & 4976 & 44746 & 44638 & 39662 & 0
\end{array}
\end{equation*}
Hence we can conclude 
\begin{proposition}
For all $3 \leq b \leq 12$, $\widehat{X}_{b,b}$ is a regular minimal surface of general type.
\end{proposition}
We notice that this result would be difficult to achieve with the techniques of \cite{BaPi12,BaPi16} since both the minimality criteria there {\it e.g.} \cite{BaPi12}*{Proposition 4.7} and \cite{BaPi16}*{Lemma 6.9} require that at most two of the exceptional divisors of the resolution of the singularities of the quotient model have self-intersection different to $-2$ and $-3$, whereas in the last example we have $12$ curves of self-intersection $-12$.

This disproves the conjecture \cite{BaPi16}*{Conjecture 1.5}, proved in \cite{BaPi16} for surfaces with $p_g=0$. Indeed, all these surfaces have invariant $\gamma$ (\cite{BaPi16}*{Definition 2.3}) equal to zero, so $p_g+\gamma =p_g \neq 0$, whereas  \cite{BaPi16}*{Conjecture 1.5} suggests that all minimal product-quotient surfaces should have $p_g+\gamma=0$.

%%%%%%%%%%%%%%%%%%%%%%%%%%%%%%%%%%%%%%%%%%%%%%%%%%%%%%%%%%%%%%%%%%%%%%%

%=========================================================================

\end{document}